\documentclass[reqno]{amsart}
\usepackage{CJK}
\usepackage{hyperref}
\usepackage{cite}
\usepackage{amsmath}
\usepackage{mathrsfs}
\usepackage{amssymb}
\usepackage{amsfonts}
\def\dive{\operatorname{div}}

\DeclareMathOperator*{\osc}{osc}

\makeatletter
\@namedef{subjclassname@2020}{%
	\textup{2020} Mathematics Subject Classification}
\makeatother

\numberwithin{equation}{section}

\newtheorem{theorem}{Theorem}[section]
\newtheorem{lemma}[theorem]{Lemma}
\newtheorem{definition}[theorem]{Definition}

\newtheorem{proposition}[theorem]{Proposition}
\newtheorem{remark}[theorem]{Remark}
\newtheorem{corollary}[theorem]{Corollary}

\allowdisplaybreaks

\newcommand{ \mint }{ {\int\hspace{-0.38cm}-}}

\begin{document}
	
	\title[\hfil Regularity theory for nonlocal equations\dots] {Regularity theory for nonlocal equations with general growth in the Heisenberg group}
	
	\author[Y. Fang and C. Zhang  \hfil \hfilneg]{Yuzhou Fang and  Chao Zhang$^*$}
	
	\thanks{$^*$Corresponding author.}
	
	\address{Yuzhou Fang \hfill\break School of Mathematics, Harbin Institute of Technology, Harbin 150001, China}
	\email{18b912036@hit.edu.cn}

	\address{Chao Zhang  \hfill\break School of Mathematics and Institute for Advanced Study in Mathematics, Harbin Institute of Technology, Harbin 150001, China}
	\email{czhangmath@hit.edu.cn}

	
	\maketitle
	
	\begin{abstract}
	We deal with a wide class of generalized nonlocal $p$-Laplace equations, so-called nonlocal $G$-Laplace equations, in the Heisenberg framework. Under natural hypotheses on the $N$-function $G$, we provide a unified approach to investigate in the spirit of De Giorgi-Nash-Moser theory, some local properties of weak solutions to such kind of problems, involving boundedness, H\"{o}lder continuity and Harnack inequality. To this end, an improved nonlocal Caccioppoli-type estimate as the main auxiliary ingredient is exploited several times. 
	\end{abstract}
	
	\section{Introduction}
	\label{sec-0}
	
	The objective of this paper is to develop local regularity theory for the weak solutions of a very general class of nonlocal integro-differential problems with non-standard growth in the Heisenberg group $\mathbb{H}^n$. The kind of equations we are addressing are the following:
	\begin{equation}
		\mathcal{L}u(\xi)
	    =0\quad \text{in } \Omega
	    \label{main}
	\end{equation}
	with the operator $\mathcal{L}$ given by
	\begin{equation*}
		\mathcal{L}u(\xi):=\mathrm{P.V.}\int_{\mathbb{H}^n}
		g\left( \frac{|u(\xi)-u(\eta)|}{|\eta^{-1}\circ\xi|^s_{\mathbb{H}^n}}\right)
		\frac{u(\xi)-u(\eta)}{|u(\xi)-u(\eta)|}
		\frac{d\eta}{|\eta^{-1}\circ\xi|^{Q+s}_{\mathbb{H}^n}},
	\end{equation*}
	where $s\in (0,1)$, $\Omega$ is a bounded domain in $\mathbb{H}^n$ for $n\ge 1$, and $Q=2n+2$ stands for the homogeneous dimension of $\mathbb{H}^n$. Here the symbols $|\cdot|_{\mathbb{H}^n}$ and $\mathrm{P.V.}$ denote separately the standard Heisenberg norm and ``in the principal value sense". It is noteworthy that the function $g:\left[ 0,\infty\right) \rightarrow\left[ 0,\infty\right)$ is continuous and strictly increasing such that $g(0)=0$, $\lim\limits_{t\rightarrow\infty}g(t)=\infty$ and
	\begin{equation}
		1<p\le \frac{tg(t)}{G(t)}\le q<\infty\quad \text{with }  G(t):=\int_{0}^{t}g(\tau)d\tau,
		\label{02}
	\end{equation}
	where $G(\cdot)$ is an $N$-function that carries the $\Delta_2$ and $\nabla_2$ conditions (see Section \ref{sec-1}). Let us point out that the necessary condition \eqref{02} firstly appeared in the work of Lieberman \cite{Lie91}, where the author established the $C^{0,\alpha}$-continuity of weak solutions to \eqref{04} below. Several representative samples, satisfying the requirements above, incorporate the power case $g(t)=t^{p-1}$, the limiting case $g(t)=t^{p-1}\ln(e+t)$ as well as $g(t)=t^{p-1}+t^{q-1}$.

Let us first pay attention to the $p$-fractional subLaplacian equations, namely, the particular scenario that $g(t)=t^{p-1}$ in \eqref{main}, which arise from many distinguish contexts, such as phase transition problems, ferromagnetic analysis, image segmentation models, quantum mechanics and so on. Integro-differential equations of this type can be seen as a generalization of the fractional subLaplacian in the Heisenberg group at this time, whose integral explanation presented by \cite{RT16}, defined explicitly over the proper fractional Sobolev space $H^s(\mathbb{H}^n)$ with $s\in(0,1)$ as below,
\begin{equation}\label{03}
\left(-\Delta_{\mathbb{H}^n}\right)^su(\xi):=
C(n,s)\lim_{r\rightarrow0^+}\int_{{\mathbb{H}^n}\setminus B_r(\xi)} {\frac{{u\left( \xi  \right) - u\left( \eta  \right)}}{{| {{\eta ^{ - 1}} \circ \xi } |_{{\mathbb{H}^n}}^{Q + 2s}}}\,d\eta },   \quad \xi\in\mathbb{H}^n,
\end{equation}
where $C(n,s)$ is a positive constant. A series of theory on the fractional subLaplacian operators has been investigated during the past decade: Hardy and uncertainty inequalities on general stratified Lie groups \cite{CCR15}, Harnack estimates and H\"{o}lder results in Carnot groups \cite{FF15}, Sobolev and Morrey-type embeddings for fractional order Sobolev spaces \cite{AM18} together with Liouville-type theorems \cite{CT16}. We refer to \cite{FMPPS18,FGMT15,GT22,GT21} and references therein for more results in the linear situation.

For what concerns the general nonlinear counterpart of \eqref{03}, that is, the $p$-growth scenario for $p\in(1,\infty)$, several interesting properties, especially on the regularity theory, have been established gradually in very recent years. In this respect, we would like to mention that Manfredini, Palatucci, Picinini and Polidoro \cite{MPP23} demonstrated the interior boundedness along with H\"{o}lder continuity for the weak solutions of the fractional $p$-subLaplacian equations in the Heisenberg group $\mathbb{H}^n$. Correspondingly, nonlocal Harnack inequalities regarding this kind of problems were proven in \cite{PP22}, where the authors considered the asymptotic behaviour for fractional linear operators as well. In addition, when it comes to the obstacle problems related to the nonlocal $p$-subLaplacian, Picinini \cite{P22} obtained the solvability, semicontinuity, boundedness as well as H\"{o}lder continuity up to the boundary for weak solutions. To some extent, the aforementioned results extend the analogues of the fractional Euclidean framework in \cite{DKP14,DKP16,KKP16,KKP17} to the Heisenberg setting. Some extra fundamental functional inequalities and quantitative estimates could be found in \cite{KS18,KS20,PP23}.

The local counterpart of Eq. \eqref{main} constitutes naturally the classical quasilinear elliptic equations in divergence form with general growth
\begin{equation}
\label{04}
-\dive\left(g(|\nabla u|)\frac{\nabla u}{|\nabla u|}\right)=0 \quad\text{with } g(t)=G'(t),
\end{equation}
for which Mukherjee \cite{Muk23} showed, under the requirement \eqref{02}, the Harnack estimates, $C^{0,\alpha}$ and further $C^{1,\alpha}$ regularity of weak solutions in the setting of Heisenberg group. On the other hand, within the Euclidean framework, the so-called $G$-Laplace equation \eqref{04} have also been attracting a great deal of attention recently. The regularity for the $(p,q)$-growth scenario is initially developed by the celebrated works of Marcellini \cite{Mar89,Mar91}. Rather comprehensive literature discussed the relevant regularity theory on Eq. \eqref{04} or related functionals, for instance, \cite{BM20,HHL21,Bar15} and references therein. 
Along the regularity research on the (local) $G$-Laplace equations, the natural topics of regularity properties for the corresponding nonlocal equations in the Euclidean space have been considered intensively in recent years. For example, Giacomoni, Kumar and Sreenadh \cite{GKS22} derived, for the nonhomogeneous case $g(t)\thickapprox t^{p-1}+t^{q-1}$, Harnack inequality and global H\"{o}lder continuity; see also \cite{GKS23} for boundary regularity and Hopf type maximum principle together with strong comparison principle in the superquadratic case $2\le p,q<\infty$. Additionally, by means of De Giorgi classes, Chaker, Kim and Weidner \cite{CKW22} proved the interior H\"{o}lder continuity for the corresponding equations and functionals, nonetheless where they imposed more restricted conditions on $G$ such that $t^p\lesssim G(t)$ and $q<\frac{np}{n-sp}$ in \eqref{02}, and required $0<s_0\le s<1$ for the robustness of their results as $s\rightarrow1^-$; subsequently, in a different manner such local behaviour was established under the assumption \eqref{02} through applying the Caccioppoli-type inequality and logarithmic-type estimate in \cite{BKO23}. More recently, we \cite{FZ} concluded a nonlocal Harnack-type inequality for \eqref{main} with the structure preconditions \eqref{02} and $G(t\tau)\lesssim G(t)G(\tau)$ in the setup of Euclidean space, and then the papers \cite{CKW23,BKS} improved this result by removing the second condition and obtaining a full Harnack inequality. Concerning more abundant studies for the nonlocal problems exhibiting non-standard growth feature, encompassing also double phase case, one could refer to \cite{DeFP19,FZ23,BOS22,DeFM,SV22} and references therein.

As has been mentioned above, a lot of work concerning the nonlocal problems with general growth has been developed in the context of the Euclidean space, but not very much has been done in the framework of Heisenberg group. In particular, the authors  and Zhang  \cite{FZZ} focused on interior regularity for the nonlocal double phase equations in the Heisenberg situation. The aim of this paper is to develop a unified approach to the regularity theory of such problems that allows us to establish the boundedness, H\"{o}lder continuity and full Harnack inequalities for the weak solutions to Eq. \eqref{main} at the same time. Our argument is more direct and transparent than those used in the previous ones even in the Euclidean setting, which mainly relied on an energy estimate--Caccioppoli inequality--in Proposition \ref{A}, avoiding the use of any logarithmic-type estimates and only requiring the basic structural hypothesis \eqref{02}. Indeed, we have to overcome the challenges stemming not only from the inhomogeneous growth of the function $G$ not carried by the $p$-subLaplacian equations, but also from noneuclidean geometrical structure and nonlocal characteristic of the integro-differential operator $\mathcal L$. Therefore, more careful systematic analyses are needed to deal with the complexity from such a nonlocal problem with Orlicz growth.

Before presenting our main results, we introduce a tail space
$$
L^g_s(\mathbb{H}^n):=\left\{u:\mathbb{H}^n\rightarrow\mathbb{R} \text{ is measurable}:\int_{\mathbb{H}^n}g\left(\frac{|u(\xi)|}{(1+|\xi|_{\mathbb{H}^n})^s}\right)\frac{d\xi}{(1+|\xi|_{\mathbb{H}^n})^{Q+s}}<\infty\right\}
$$
and the corresponding nonlocal tail in the ball $B_r(\xi_0)$ is defined as
	\begin{equation*}
		\text{Tail}(u;\xi_0,r)=
		\int_{\mathbb{H}^n\setminus B_r(\xi_0)}
		g\left( \frac{|u(\xi)|}{|{\xi_0}^{-1}\circ\xi|^s_{\mathbb{H}^n}}\right)
		\frac{d\xi}{|{\xi_0}^{-1}\circ\xi|^{Q+s}_{\mathbb{H}^n}}.
	\end{equation*}
Observe that $\mathrm{Tail}(u;\xi_0,r)<\infty$ for any $\xi_0\in\mathbb{H}^n$ and $r>0$, when $u$ belongs to $L^g_s(\mathbb{H}^n)$. For no ambiguity, we stress $g^{-1}, G^{-1}$ mean separately the inverses of $g, G$. 

Now we are in a position to state in turn the main contributions of our work. Throughout the forthcoming three theorems we always suppose the condition \eqref{02} is in force. The first result describes the local boundedness of weak solutions.

\begin{theorem}[Local boundedness]
		\label{C}
		Let $u\in H\mathbb{W}^{s,G}(\Omega)\cap L^g_s(\mathbb{H}^n)$ be a weak subsolution of \eqref{main} and $B_r:=B_r(\xi_0)\subset\subset\Omega$. Then for every $\delta\in(0,1]$ it holds that
		\begin{equation*}
			\sup_{B_{r/2}}u\le Cr^sG^{-1}\left( \delta^{\frac{\theta}{1-\theta}}\mint_{B_{r}}G\left( \frac{u_+}{r^s}\right) d\xi\right)
			+r^sg^{-1}\left( \delta(r/2)^s\mathrm{Tail}(u_+;\xi_0,r/2)\right),
		\end{equation*}
		where $\theta>1$ is from Lemma \ref{B} and $C>0$ depends only on $n,p,q,s$.
	\end{theorem}

In the theorem above, the presence of the parameter $\delta$ allows an interpolation between the local and nonlocal terms, which plays a prominent role in the proof of Harnack inequality below. The second one is the desired interior H\"{o}lder regularity for the weak solutions.	

\begin{theorem}[H\"{o}lder continuity]
		\label{thm2}
		Let $u\in H\mathbb W^{s,G}(\Omega)\cap L^g_s(\mathbb{H}^n)$ be a weak solution to \eqref{main} with local boundedness. Then $u$ is locally of the class $C^{0,\alpha}(\Omega)$ with some $\alpha\in(0,1)$ determined only by the structural constants $n,p,q,s$. To be precise, for any ball $B_R(\xi_0)\subset\subset\Omega$, there exists a $C>1$ depending also on $n,p,q,s$ such that
		\begin{equation*}
			\osc\limits_{B_r(\xi_0)}u\le C\left(\frac{r}{R}\right)^\alpha \left[\sup_{B_R(\xi_0)}|u|
			+R^sg^{-1}(R^s\mathrm{Tail}(u;\xi_0,R))\right]
		\end{equation*}
		with $r\in(0,R]$.
\end{theorem}	

Putting together Theorems \ref{C} and \ref{thm2}, we could have the H\"{o}lder continuity under the natural condition \eqref{02} (without boundedness assumption). We would like to mention that in order to get the oscillation reduction, the pivotal difficulties consist in controlling properly the nonlocal tail of solutions at every iteration step, which is a new element carried by the nonlocality of the operator $\mathcal{L}$, compared to the local case. This demands us to construct very carefully the sequence of shrinking balls $\{B_{r_j}(\xi_0)\}$ and the geometric sequence $\{\omega_j\}$ such that
\begin{equation}\label{05}
\osc\limits_{B_{r_j}(\xi_0)}u\le\omega_j.
\end{equation}
Generally speaking, if the relation \eqref{05} is inferred, we shall show the accurate estimate of tail
\begin{equation}\label{06}
r^s_j\mathrm{Tail}((u-\nu^\pm_j)_\pm;\xi_0,r_j)\le g\left(C\frac{\omega_j}{r^s_j}\right)
\end{equation}
with $\nu_j^+,\nu^-_j$ representing the supremum and infimum of $u$ in $B_{r_j}(\xi_0)$ respectively. Further, the inequality \eqref{06} leads to the improvement of oscillation in $B_{r_{j+1}}(\xi_0)$, and we can continue this process step by step. The last one provides a full Harnack inequality for Eq. \eqref{main}.
	
\begin{theorem}[Harnack inequality]
			\label{K}
			Let $u\in H\mathbb W^{s,G}(\Omega)\cap L^g_s(\mathbb{H}^n)$ be a weak solution to \eqref{main}. For any ball $B_{4R}(\xi_0)\subset\subset\Omega$, let also $u$ be nonnegative in $B_{4R}(\xi_0)$. Then we infer the following estimate
			\begin{equation}\label{07}
				\sup_{B_R(\xi_0)}u\le C\left( \inf_{B_R(\xi_0)}u+R^sg^{-1}\left( R^s\mathrm{Tail}(u_-;\xi_0,R)\right) \right)
			\end{equation}
			with the constant $C>0$ depending only on $n,p,q,s$.
\end{theorem}

Two direct components of this result are the upper bound estimate (Theorem \ref{C}) and the weak Harnack inequality (Lemma \ref{I}), where the latter was derived by the so-called expansion of positivity in Section \ref{sec-3} and Krylov-Sofonov covering lemma. Observe that, whenever the weak solutions of \eqref{main} are nonnegative in the whole space $\mathbb{H}^n$, \eqref{07} is reduced to the standard elliptic Harnack inequality that is however showed to fail if nonnegativity of solution is only assumed in local domain in \cite{Kas}. Besides, if $g(t)=t^{p-1}$, i.e., Eq. \eqref{main} turns into the fractional $p$-subLaplacian, then the nonlocal term
$$
R^sg^{-1}( R^s\mathrm{Tail}(u_-;\xi_0,R)=\left(R^{sp}\int_{\mathbb{H}^n\setminus B_R(\xi_0)}
		\frac{u^{p-1}_-(\xi)}{|{\xi_0}^{-1}\circ\xi|^{Q+ps}_{\mathbb{H}^n}}\,d\xi\right)^{\frac{1}{p-1}}.
$$
At this stage, our result \eqref{07} is simplified to the Harnack inequality obtained in \cite{PP22}.

The paper is organized as follows. In Section \ref{sec-1}, we give the notion of weak solutions to Eq. \eqref{main}, and collect some notations and auxiliary tools to be used later. Section \ref{sec-2} is devoted to deducing the improved Caccioppoli estimate and then the local boundedness, while we show H\"{o}lder continuity for weak solutions in Section \ref{sec-3}. Finally, we prove the Harnack inequality in Section \ref{sec-4}.

	\section{Preliminaries}
	\label{sec-1}
	
	In this section, we shall give some basic inequalities, state the notions of some functional spaces and weak solutions, and then provide an iteration lemma.

The Euclidean space ${{\mathbb{R}}^{2n + 1}}\;(n \ge 1)$ with the group multiplication
\[\xi \circ \eta = \left( {{x_1} + {y_1},{x_2} + {y_2}, \cdots ,{x_{2n}} + {y_{2n}},\tau + \tau' + \frac{1}{2}\sum\limits_{i = 1}^n {\left( {{x_i}{y_{n + i}} - {x_{n + i}}{y_i}} \right)} } \right),\]
where $\xi = \left( {{x_1},{x_2}, \cdots ,{x_{2n}},\tau} \right),$ $\eta = \left( {{y_1},{y_2}, \cdots ,{y_{2n}},\tau'} \right) \in {{\mathbb{R}}^{2n+1}},$ leads to the Heisenberg group ${{\mathbb{H}}^n}$. The left invariant vector field on ${{\mathbb{H}}^n}$ is of the form
\begin{equation*}\label{eq112}
  {X_i} = {\partial _{{x_i}}} - \frac{{{x_{n + i}}}}{2}{\partial _\tau},\;{X_{n + i}} = {\partial _{{x_{n + i}}}} + \frac{{{x_i}}}{2}{\partial _\tau},\quad 1 \le i \le n
\end{equation*}
and a non-trivial commutator is
\begin{equation*}\label{eq113}
  T = {\partial _\tau} = \left[ {{X_i},{X_{n + i}}} \right] = {X_i}{X_{n + i}} - {X_{n + i}}{X_i},~1 \le i \le n.
\end{equation*}
We call that ${X_1},{X_2}, \cdots ,{X_{2n}}$ are the horizontal vector fields on ${{\mathbb{H}}^n}$ and $T$ the vertical vector field. For a smooth function $u$ on ${{\mathbb{H}}^n}$, the subgradient is defined by
\begin{equation*}\label{eq114}
{\nabla _{\mathbb H^n}}u = \left( {{X_1}u,{X_2}u, \cdots ,{X_{2n}}u} \right).
\end{equation*}

The Haar measure in ${{\mathbb{H}}^n}$ is equivalent to the Lebesgue measure in ${{\mathbb{R}}^{2n+1}}$. We denote by $\left| E \right|$ the Lebesgue measure of a measurable set $E \subset {{\mathbb{H}}^n}$. If $f\in L^1(E)$ and $E\subset\mathbb{H}^n$ is a measurable subset with positive measure $0<|E|<\infty$, we write
	$$
	(f)_E:=\mint_Ef(\xi)\,d\xi=\frac{1}{|E|}\int_Ef(\xi)\,d\xi.
	$$
For $\xi =( {{x_1},{x_2}, \cdots ,{x_{2n}},\tau})\in\mathbb{H}^n$, we define its standard homogeneous norm as
\[{|\xi|_{{{\mathbb{H}}^n}}} = {\left( {{{\left( {\sum\limits_{i = 1}^{2n} {{x_i}^2} } \right)}^2} + {\tau^2}} \right)^{\frac{1}{4}}}.\]
We denote the ball $B_r(\xi_0)$ with center $\xi_0\in\mathbb{H}^n$ and radius $r>0$ by
\begin{equation*}
  {B_r }\left( \xi_0 \right) = \left\{ {\xi \in {{\mathbb{H}}^n}:|\xi_0^{ - 1}\circ \xi |_{\mathbb{H}^n}< r } \right\}.
\end{equation*}
Whenever not important or clear from the context, we shall drop the center as follows: $B_r:=B_r( \xi_0)$.

	The function $G:[0,\infty)\rightarrow[0,\infty)$ is an $N$-function, if it is convex, increasing, and fulfills that
	$$
	G(0)=0, \quad \lim_{t\rightarrow0+}\frac{G(t)}{t}=0 \quad\text{and} \quad \lim_{t\rightarrow\infty}\frac{G(t)}{t}=\infty.
	$$
	The conjugate function of this $G$ is given as
	$$
	G^*(t)=\sup_{\tau\geq0}\{\tau t-G(\tau)\}.
	$$
   From \eqref{02}, we provide several known inequalities that are utilized later:
	\begin{itemize}
		\item[($a$)] for $t\in[0,\infty)$,
		\begin{equation}
		\label{1-1}
		\begin{cases}
		a^qG(t)\leq G(at)\leq a^pG(t) & \text{\textmd{if }} a\in(0,1),\\[2mm]
		a^pG(t)\leq G(at)\leq a^qG(t) & \text{\textmd{if }} a\in(1,\infty)
		\end{cases}
		\end{equation}
		and
		\begin{equation}
		\label{1-2}
		\begin{cases}
		a^{p'}G^*(t)\leq G^*(at)\leq a^{q'}G^*(t) & \text{\textmd{if }} a\in(0,1),\\[2mm]
		a^{q'}G^*(t)\leq G^*(at)\leq a^{p'}G^*(t) & \text{\textmd{if }} a\in(1,\infty),
		\end{cases}
		\end{equation}
		where $p',q'$ are the H\"{o}lder conjugates of $p,q$.
		
		\item[($b$)] Young's inequality with $\epsilon\in(0,1]$
		\begin{equation}
		\label{1-3}
		t\tau\leq \epsilon^{1-q}G(t)+\epsilon G^*(\tau), \quad t,\tau\geq0.
		\end{equation}
		
		\item[($c$)] for $t,\tau\geq0$,
		\begin{equation}
		\label{1-4}
		G^*(g(t))\leq (q-1)G(t),
		\end{equation}
		and
		\begin{equation}
		\label{1-5}
		2^{-1}(G(t)+G(\tau))\leq G(t+\tau)\leq 2^{q-1}(G(t)+G(\tau)).
		\end{equation}
		
	\end{itemize}
	
	Moreover, the $N$-function $G$ meets $\Delta_2$ and $\nabla_2$ conditions (see \cite[Proposition 2.3]{MR08}):
	\begin{itemize}
		
		\item[($\Delta_2$)] there exists $\mu>1$ such that $G(2t)\leq \mu G(t)$ for $t\geq0$;
		
		\smallskip
		
		\item[($\nabla_2$)] there exists $\nu>1$ such that $G(t)\leq \frac{1}{2\nu}G(\nu t)$ for $t\geq 0$,
		
	\end{itemize}
	where $\mu,\nu$ depend on $p,q$. In fact, the condition $\nabla_2$ is just $\Delta_2$ applied to $G^*$.
	
	We next introduce the notions of fractional Orlicz-Sobolev spaces in the Heisenberg framework. For a given domain $\Omega\subset\mathbb{H}^n$ and an $N$-function $G$ with the $\Delta_2$ and $\nabla_2$ conditions, the Orlicz space $L^G(\Omega)$ is denoted by
	$$
	L^G(\Omega)=\left\{u:\Omega\rightarrow\mathbb{R} \text{ is measurable}: \int_\Omega G(|u(\xi)|)\,d\xi<\infty\right\}
	$$
	with the Luxemburg norm
	$$
	\|u\|_{L^G(\Omega)}=\inf\left\{\lambda>0: \int_\Omega G\left(\frac{|u(\xi)|}{\lambda}\right)\,d\xi\leq1\right\}.
	$$
	The fractional Orlicz-Sobolev space $HW^{s,G}(\Omega)$ ($s\in(0,1)$) is defined as
	$$
	HW^{s,G}(\Omega)=\left\{u\in L^G(\Omega):\int_\Omega\int_\Omega G\left(\frac{|u(\xi)-u(\eta)|}{|\eta^{-1}\circ\xi|^s_{\mathbb{H}^n}}\right)\,\frac{d\xi d\eta}{|\eta^{-1}\circ\xi|^Q_{\mathbb{H}^n}}<\infty\right\}
	$$
	equipped with the norm
	$$
	\|u\|_{HW^{s,G}(\Omega)}=\|u\|_{L^G(\Omega)}+[u]_{s,G,\Omega},
	$$
	where $[u]_{s,G,\Omega}$ represents the Gagliardo semi-norm
	$$
	[u]_{s,G,\Omega}=\inf\left\{\lambda>0:\int_\Omega\int_\Omega G\left(\frac{|u(\xi)-u(\eta)|}{\lambda|\eta^{-1}\circ\xi|^s_{\mathbb{H}^n}}\right)\,\frac{d\xi d\eta}{|\eta^{-1}\circ\xi|^Q_{\mathbb{H}^n}}\leq1\right\}.
	$$
	
	Let $\mathcal{C}_{\Omega}:= (\Omega\times\mathbb{H}^n)\cup(\mathbb{H}^n\times\Omega)$. For a measurable function $u$ in $\mathbb H^n$, we define 
	\begin{align*}
	H\mathbb{W}^{s,G}(\Omega)=\left\{u\big |_{\Omega}\in L^G(\Omega) : \iint_{\mathcal{C}_{\Omega}}G\left(\frac{|u(\xi)-u(\eta)|}{|\eta^{-1}\circ\xi|^s_{\mathbb{H}^n}}\right)\,\frac{d\xi d\eta}{|\eta^{-1}\circ\xi|^Q_{\mathbb{H}^n}}<\infty\right\},
	\end{align*}
	which is the function space weak solutions to \eqref{main} belong to. In the sequel, denote by $C$ a generic positive constant that may vary from line to line. Relevant dependencies on parameters will be explained by parentheses, i.e., $C\equiv C(n,p,q)$ means $C$ depends on $n,p,q$.
	
	\medskip
	
	Now we present the definition of weak solutions to \eqref{main}.
	\begin{definition}
		A function $u\in H\mathbb{W}^{s,G}(\Omega)$ is called weak solution (supersolution or subsolution) to Eq. \eqref{main}, if
		\begin{equation*}
		\iint_{\mathcal{C}_{\Omega}}g\left(\frac{|u(\xi)-u(\eta)|}{|\eta^{-1}\circ\xi|^s_{\mathbb{H}^n}}\right)
		\frac{u(\xi)-u(\eta)}{|u(\xi)-u(\eta)|}(\psi(\xi)-\psi(\eta))\frac{d\xi d\eta}{|\eta^{-1}\circ\xi|^{Q+s}_{\mathbb{H}^n}}=0 \ (\geq \text{ or } \le)0
		\end{equation*}
		for every $\psi\in H\mathbb{W}^{s,G}(\Omega)$ ($0\le\psi\in H\mathbb{W}^{s,G}(\Omega)$) with compact support in $\Omega$. 
	\end{definition}

	We end this part by the following iteration tool, which can be found in \cite[Lemma 1.1]{GG82}, playing a significant role in the proof of Theorem \ref{K}, Harnack inequality.
	\begin{lemma}
		\label{2.7Lemma}
		Suppose $f(t)$ is a bounded nonnegative function defined in $0\le T_0\le t\le T_1$. When for $T_0\le \sigma< \tau \le T_1$, we have
		\begin{equation*}
			f(\sigma)\le \iota f(\tau)+C_1(\tau-\sigma)^{-\gamma}+C_2
		\end{equation*}
		with $\iota\in (0,1)$ and $\gamma,C_1,C_2$ being nonnegative constants, this there is a number $C>0$, depending only upon $\gamma$ and $\iota$, such that, for each $T_0\le \rho <r\le T_1$, it holds
		\begin{equation*}
			f(\rho)\le C\left[ C_1(r-\rho)^{-\gamma}+C_2\right].
		\end{equation*}
	\end{lemma}

	\section{Energy estimates and boundedness}
	\label{sec-2}
	
	This section is devoted to establishing a improved Caccioppoli estimate involving all the information needed to demonstrate (local) regularity properties for weak solutions to \eqref{main} such as boundedness, H\"{o}lder continuity together with Harnack estimates. For the notations $\pm$ and $\mp$ blow, let us point out that subsolution and supersolution always correspond to the upper sign and lower sign, respectively. Besides, for convenience, we introduce several sets
		\begin{equation*}
			 A^+(k,r)=B_r\cap \left\lbrace u>k\right\rbrace ,\;
			A^-(k,r)=B_r\cap \left\lbrace u<k\right\rbrace
\end{equation*}
and
\begin{equation*}
			 A^+(k)=\left\lbrace u>k\right\rbrace,\;
			A^-(k)=\left\lbrace u<k\right\rbrace.
		\end{equation*}
	
	\begin{proposition}
		\label{A}
		Let u be a weak subsolution $(super\text{-})$ to Eq. \eqref{main} and $B_r:=B_r(\xi_0)\subset\Omega$. Then for any $0<\rho<r$, it holds that
		\begin{equation*}
			\begin{aligned}
				&\quad\int_{B_\rho}\int_{B_\rho}
				G\left( \frac{|w_\pm(\xi)-w_\pm(\eta)|}{|\eta^{-1}\circ\xi|^s_{\mathbb{H}^n}} \right)
				\frac{d\xi d\eta }{|\eta^{-1}\circ\xi|^{Q}_{\mathbb{H}^n}}\\
				&\quad+\int_{B_\rho}w_\pm(\xi)\int_{\mathbb{H}^n}
				g\left( \frac{w_\mp(\eta)}{|\eta^{-1}\circ\xi|^s_{\mathbb{H}^n}}\right)
				\frac{d\eta}{|\eta^{-1}\circ\xi|^{Q+s}_{\mathbb{H}^n}}\,d\xi\\
				&\le C\left( \frac{r}{r-\rho}\right)^q
				\int_{B_r} G\left( \frac{w_\pm}{r^s}\right) d\xi
				+C\left( \frac{r}{r-\rho}\right)^{Q+sq} ||w_\pm||_{L^1(B_r)}\mathrm{Tail}(w_\pm;\xi_0,r)
			\end{aligned}
		\end{equation*}
		with $C$ depending only on n,p,q,s. Here $w_\pm:=(u-k)_\pm$ for $k\in\mathbb{R}$.
	\end{proposition}
	
	\begin{proof}
		We prove this statement for subsolutions, because the situation of supersolutions could be treated in a specular manner. Let $\phi\in C_0^\infty(B_r(\xi_0))$ be a cut-off function such that
		\begin{equation*}
			0\le \phi \le 1,\
			\phi\equiv 1 \ \
			\text{in}\ B_\rho ,\
			\phi=0\ \text{on}\ B_r\setminus B_{\frac{r+\rho}{2}}\quad
			\text{and}\quad
			|\nabla_{\mathbb{H}^n}\phi|\le \frac{C}{r-\rho}\ \text{in}\ B_r.
		\end{equation*}
		Take $\varphi:=\phi^q w_+=\phi^q(u-k)_+$ as a test function in the weak formulation of subsolution and then derive
		\begin{align}
				\label{1.1}
				0&\ge \int_{B_r}\int_{B_r}
				g\left( \frac{|u(\xi)-u(\eta)|}{|\eta^{-1}\circ\xi|^s_{\mathbb{H}^n}}\right)
				\frac{u(\xi)-u(\eta)}{|u(\xi)-u(\eta)|}
				\frac{(\phi^q w_+)(\xi)-(\phi^q w_+)(\eta)}{|\eta^{-1}\circ\xi|^{Q+s}_{\mathbb{H}^n}}
				d\xi d\eta\nonumber\\
				&\quad+2\int_{\mathbb{H}^n\setminus B_r}\int_{B_r}
				g\left( \frac{|u(\xi)-u(\eta)|}{|\eta^{-1}\circ\xi|^s_{\mathbb{H}^n}}\right)
				\frac{u(\xi)-u(\eta)}{|u(\xi)-u(\eta)|}
				\frac{\phi^q w_+(\xi)}{|\eta^{-1}\circ\xi|^{Q+s}_{\mathbb{H}^n}}
				d\xi d\eta\nonumber\\
				&:=I_1+2I_2.
		\end{align}\par
		We first consider the contributions from $B_r\times B_r$. Set
		\begin{equation*}
			F(\xi,\eta)=g\left( \frac{|u(\xi)-u(\eta)|}{|\eta^{-1}\circ\xi|^s_{\mathbb{H}^n}}\right)
			\frac{u(\xi)-u(\eta)}{|u(\xi)-u(\eta)|}
			\frac{\left( \phi^q w_+\right) (\xi)-\left( \phi^q w_+\right) (\eta)}{|\eta^{-1}\circ\xi|^{s}_{\mathbb{H}^n}}.
		\end{equation*}
		If $\xi,\eta\notin A^+(k)$, then $F(\xi,\eta)=0$. If $\xi\in A^+(k,r),\eta\in B_r\setminus A^+(k,r)$, then by \eqref{02} and the monotonicity of $g$,
		\begin{equation*}
			\begin{aligned}
				F(\xi,\eta)&=g\left( \frac{|w_+(\xi)+w_-(\eta)|}{|\eta^{-1}\circ\xi|^s_{\mathbb{H}^n}}\right)
				\frac{w_+(\xi)\phi^q (\xi)}{|\eta^{-1}\circ\xi|^{s}_{\mathbb{H}^n}}\\
				&\ge \frac{1}{2}
				\left[
				g\left( \frac{w_+(\xi)}{|\eta^{-1}\circ\xi|^s_{\mathbb{H}^n}}\right)
				+g\left( \frac{w_-(\xi)}{|\eta^{-1}\circ\xi|^s_{\mathbb{H}^n}}\right)
				\right]
				\frac{w_+(\xi)\phi^q(\xi)}{|\eta^{-1}\circ\xi|^{s}_{\mathbb{H}^n}}\\
				&\ge \frac{p}{2}G\left( \frac{|w_+(\xi)-w_+(\eta)|}{|\eta^{-1}\circ\xi|^s_{\mathbb{H}^n}}\right)
				+\frac{1}{2}g\left( \frac{w_-(\eta)}{|\eta^{-1}\circ\xi|^s_{\mathbb{H}^n}}\right)
				\frac{w_+(\xi)\phi^q(\xi)}{|\eta^{-1}\circ\xi|^{s}_{\mathbb{H}^n}}.
			\end{aligned}
		\end{equation*}\par
		In the case $\xi,\eta\in A^+(k,r)$, we, without loss of generality, suppose $u(\xi)>u(\eta)$, since we can exchange the roles of $\xi$ and $\eta$ for $u(\eta)>u(\xi)$ and find $F(\xi,\eta)=0$ for $u(\eta)=u(\xi)$. Under this constrained scenario, we evaluate $F(\xi,\eta)$ for $\phi(\xi)\ge\phi(\eta)$ as below,
		\begin{equation*}
			\begin{aligned}
				F(\xi,\eta)&=g\left( \frac{w_+(\xi)-w_+(\eta)}{|\eta^{-1}\circ\xi|^s_{\mathbb{H}^n}}\right)
				\frac{\left( \phi^qw_+\right) (\xi)-\left( \phi^qw_+\right) (\eta)}{|\eta^{-1}\circ\xi|^{s}_{\mathbb{H}^n}}\\
				&\ge g\left( \frac{w_+(\xi)-w_+(\eta)}{|\eta^{-1}\circ\xi|^s_{\mathbb{H}^n}}\right)
				\frac{w_+(\xi)-w_+(\eta)}{|\eta^{-1}\circ\xi|^{s}_{\mathbb{H}^n}}\phi^q(\xi)\\
				&\ge pG\left( \frac{w_+(\xi)-w_+(\eta)}{|\eta^{-1}\circ\xi|^{s}_{\mathbb{H}^n}}\right) \phi^q(\xi).
			\end{aligned}
		\end{equation*}
		As for $\phi(\xi)<\phi(\eta)$, we have
			\begin{align}
				\label{1.2}
				&\quad F(\xi,\eta)\nonumber\\
				&=g\left( \frac{w_+(\xi)-w_+(\eta)}{|\eta^{-1}\circ\xi|^s_{\mathbb{H}^n}}\right)
				\left[
				\frac{w_+(\xi)-w_+(\eta)}{|\eta^{-1}\circ\xi|^{s}_{\mathbb{H}^n}}\phi^q(\eta)
				+\frac{\phi^q(\xi)-\phi^q(\eta)}{|\eta^{-1}\circ\xi|^{s}_{\mathbb{H}^n}}w_+(\xi)
				\right] \nonumber\\
				&\ge pG\left( \frac{w_+(\xi)-w_+(\eta)}{|\eta^{-1}\circ\xi|^{s}_{\mathbb{H}^n}}\right) \phi^q(\eta)
				-q\phi^{q-1}(\eta)g\left( \frac{w_+(\xi)-w_+(\eta)}{|\eta^{-1}\circ\xi|^s_{\mathbb{H}^n}}\right)
				\frac{\phi(\eta)-\phi(\xi)}{|\eta^{-1}\circ\xi|^{s}_{\mathbb{H}^n}}w_+(\xi),
			\end{align}
		where we used the fact that
		\begin{equation*}
			\phi^q(\eta)-\phi^q(\xi)\le q\phi^{q-1}(\eta)\left( \phi(\eta)-\phi(\xi)\right).
		\end{equation*}
		Applying now \eqref{1-2}--\eqref{1-4} estimates
			\begin{align}
				\label{1.3}
				&\quad\phi^{q-1}(\eta)g\left( \frac{w_+(\xi)-w_+(\eta)}{|\eta^{-1}\circ\xi|^s_{\mathbb{H}^n}}\right)
				\frac{\phi(\eta)-\phi(\xi)}{|\eta^{-1}\circ\xi|^s_{\mathbb{H}^n}}w_+(\xi)\nonumber\\
				&\le \varepsilon G^*\left( \phi^{q-1}(\eta)g\left( \frac{w_+(\xi)-w_+(\eta)}{|\eta^{-1}\circ\xi|^s_{\mathbb{H}^n}}\right)\right)
				+C(\varepsilon)G\left( \frac{w_+(\xi)\left( \phi(\eta)-\phi(\xi)\right)}{|\eta^{-1}\circ\xi|^s_{\mathbb{H}^n}}\right) \nonumber\\
				&\le \varepsilon(q-1)G\left( \frac{w_+(\xi)-w_+(\eta)}{|\eta^{-1}\circ\xi|^{s}_{\mathbb{H}^n}}\right) \phi^q(\eta)
				+C(\varepsilon)G\left( \frac{w_+(\xi)\left( \phi(\eta)-\phi(\xi)\right)}{|\eta^{-1}\circ\xi|^s_{\mathbb{H}^n}}\right),
			\end{align}
		this time we choose $\varepsilon=\frac{p}{2q(q-1)}$ and combine \eqref{1.2},\eqref{1.3} to arrive at
		\begin{equation*}
			F(\xi,\eta)\ge \frac{p}{2} G\left( \frac{w_+(\xi)-w_+(\eta)}{|\eta^{-1}\circ\xi|^{s}_{\mathbb{H}^n}}\right) \phi^q(\eta)
			-C(p,q)G\left( \frac{w_+(\xi)\left( \phi(\eta)-\phi(\xi)\right)}{|\eta^{-1}\circ\xi|^s_{\mathbb{H}^n}}\right).
		\end{equation*}
		As a consequence, for $\xi,\eta\in A^+(k,r)$, there holds that
		\begin{equation*}
			\begin{aligned}
				F(\xi,\eta)&\ge \frac{p}{2} G\left( \frac{|w_+(\xi)-w_+(\eta)|}{|\eta^{-1}\circ\xi|^{s}_{\mathbb{H}^n}}\right)\text{max}\left\lbrace \phi^q(\xi),\phi^q(\eta)\right\rbrace \\
				&\quad-C(p,q)G\left( \frac{\text{max}\left\lbrace w_+(\xi),w_+(\eta)\right\rbrace |\phi(\xi)-\phi(\eta)|}{|\eta^{-1}\circ\xi|^{s}_{\mathbb{H}^n}}\right) .
			\end{aligned}
		\end{equation*}\par
		It follows from these estimates above that
			\begin{align}
				\label{1.4}
				I_1&\ge \frac{p}{2}\int_{B_r}\int_{B_r}G\left( \frac{|w_+(\xi)-w_+(\eta)|}{|\eta^{-1}\circ\xi|^{s}_{\mathbb{H}^n}}\right)
				\frac{\text{min}\left\lbrace \phi^q(\xi),\phi^q(\eta)\right\rbrace }{|\eta^{-1}\circ\xi|^{Q}_{\mathbb{H}^n}}
				d\xi d\eta\nonumber\\
				&\quad+\int_{B_{r}\setminus A^+(k,r)}\int_{A^+(k,r)}g\left( \frac{w_-(\eta)}{|\eta^{-1}\circ\xi|^{s}_{\mathbb{H}^n}}\right)
				\frac{w_+(\xi)\phi^q(\xi)}{|\eta^{-1}\circ\xi|^{Q}_{\mathbb{H}^n}}
				d\xi d\eta\nonumber\\
				&\quad-C\int_{B_{r}}\int_{B_{r}}G\left( \frac{\text{max}\left\lbrace w_+(\xi),w_+(\eta)\right\rbrace |\phi(\xi)-\phi(\eta)|}{|\eta^{-1}\circ\xi|^{s}_{\mathbb{H}^n}}\right)
				\frac{d\xi d\eta}{|\eta^{-1}\circ\xi|^{Q}_{\mathbb{H}^n}}\nonumber\\
				&\ge \frac{p}{2}\int_{B_\rho}\int_{B_\rho}G\left( \frac{|w_+(\xi)-w_+(\eta)|}{|\eta^{-1}\circ\xi|^{s}_{\mathbb{H}^n}}\right)
				\frac{d\xi d\eta}{|\eta^{-1}\circ\xi|^{Q}_{\mathbb{H}^n}}\nonumber\\
				&\quad+\int_{B_\rho}w_+(\xi)
				\left[
				\int_{B_r}g\left( \frac{w_-(\eta)}{|\eta^{-1}\circ\xi|^{s}_{\mathbb{H}^n}}\right)
				\frac{d\eta}{|\eta^{-1}\circ\xi|^{Q+s}_{\mathbb{H}^n}}
				\right]  d\xi\nonumber\\
				&\quad-C\int_{B_r}\int_{B_r}G\left( \frac{\text{max}\left\lbrace w_+(\xi),w_+(\eta)\right\rbrace |\phi(\xi)-\phi(\eta)|}{|\eta^{-1}\circ\xi|^{s}_{\mathbb{H}^n}}\right)
				\frac{d\xi d\eta}{|\eta^{-1}\circ\xi|^{Q}_{\mathbb{H}^n}}.
			\end{align}
		Now recalling the properties of $\phi$, we get
		\begin{align*}
			|\phi(\xi)-\phi(\eta)|\le |\eta^{-1}\circ\xi|_{\mathbb{H}^n}\sup_{B_r}|\nabla_{\mathbb H^n}\phi|\le \frac{C}{r-\rho}|\eta^{-1}\circ\xi|_{\mathbb{H}^n},
		\end{align*}
		then via \eqref{1-1} and \eqref{1-5}
		\begin{align}
				\label{1.5}
				&\quad\int_{B_{r}}\int_{B_{r}}G\left( \frac{\text{max}\left\lbrace w_+(\xi),w_+(\eta)\right\rbrace |\phi(\xi)-\phi(\eta)|}{|\eta^{-1}\circ\xi|^{s}_{\mathbb{H}^n}}\right)
				\frac{d\xi d\eta}{|\eta^{-1}\circ\xi|^{Q}_{\mathbb{H}^n}}\nonumber\\
				&\le C\int_{B_{r}}\int_{B_{r}}G\left( w_+(\xi)\frac{|\eta^{-1}\circ\xi|^{1-s}_{\mathbb{H}^n}}{r-\rho}\right)
				\frac{d\xi d\eta}{|\eta^{-1}\circ\xi|^{Q}_{\mathbb{H}^n}}\nonumber\\
				&=C\int_{B_{r}}\int_{B_{r}}G\left( \frac{2r}{r-\rho}\left( \frac{|\eta^{-1}\circ\xi|_{\mathbb{H}^n}}{2r}\right)^{1-s}\frac{w_+(\xi)}{(2r)^s} \right)
				\frac{d\xi d\eta}{|\eta^{-1}\circ\xi|^{Q}_{\mathbb{H}^n}}\nonumber\\
				&\le C\left( \frac{r}{r-\rho}\right)^q \int_{B_{r}}\int_{B_{r}}\left( \frac{|\eta^{-1}\circ\xi|_{\mathbb{H}^n}}{2r}\right)^{p(1-s)}
				G\left( \frac{w_+(\xi)}{r}\right) \frac{d\xi d\eta}{|\eta^{-1}\circ\xi|^{Q}_{\mathbb{H}^n}}\nonumber\\
				&\le C\left( \frac{r}{r-\rho}\right)^q \int_{B_r}G\left( \frac{w_+(\xi)}{r^s}\right)d\xi
				\int_{B_{2r}(\xi)}\frac{|\eta^{-1}\circ\xi|^{-Q+p(1-s)}_{\mathbb{H}^n}}{(2r)^{p(1-s)}}d\eta\nonumber\\
				&\le C\left( \frac{r}{r-\rho}\right)^q \int_{B_r}G\left( \frac{w_+(\xi)}{r^s}\right)d\xi,
		\end{align}
		where $C\ge1$ depends on $n,p,q,s$. Plugging \eqref{1.5} into \eqref{1.4} yields that
			\begin{align}
				\label{1.6}
				I_1&\ge \frac{p}{2}\int_{B_{\rho}}\int_{B_{\rho}}G\left( \frac{|w_+(\xi)-w_+(\eta)|}{|\eta^{-1}\circ\xi|^{s}_{\mathbb{H}^n}}\right)
				\frac{d\xi d\eta}{|\eta^{-1}\circ\xi|^{Q}_{\mathbb{H}^n}}\nonumber\\
				&\quad+\int_{B_\rho}w_+(\xi)
				\left[
				\int_{B_r}g\left( \frac{w_-(\eta)}{|\eta^{-1}\circ\xi|^{s}_{\mathbb{H}^n}}\right)
				\frac{d\eta}{|\eta^{-1}\circ\xi|^{Q+s}_{\mathbb{H}^n}}
				\right]  d\xi\nonumber\\
&\quad-C\left( \frac{r}{r-\rho}\right)^q\int_{B_r}G\left( \frac{w_+(\xi)}{r^s}\right)d\xi,
			\end{align}
		where $C\ge 1$ depends on $n,p,q,s$.

		We next deal with the nonlocal integral $I_2$,
			\begin{align}
				\label{1.7}
				I_2&=\int_{B_r}\left( \phi^qw_+\right) (\xi)\left[
				\int_{\mathbb{H}^n\setminus B_r}g\left( \frac{|u(\xi)-u(\eta)|}{|\eta^{-1}\circ\xi|^{s}_{\mathbb{H}^n}}\right)
				\frac{u(\xi)-u(\eta)}{|u(\xi)-u(\eta)|}
				\frac{d\eta}{|\eta^{-1}\circ\xi|^{Q+s}_{\mathbb{H}^n}}
				\right]
				d\xi\nonumber\\
				&=\int_{B_r}\left( \phi^qw_+\right) (\xi)\left[
				\int_{\left\lbrace u(\xi)\ge u(\eta)\right\rbrace\setminus B_r }
				g\left( \frac{|u(\xi)-u(\eta)|}{|\eta^{-1}\circ\xi|^{s}_{\mathbb{H}^n}}\right)
				\frac{u(\xi)-u(\eta)}{|u(\xi)-u(\eta)|}
				\frac{d\eta}{|\eta^{-1}\circ\xi|^{Q+s}_{\mathbb{H}^n}}
				\right]
				d\xi\nonumber\\
				&\quad+\int_{B_r}\left( \phi^qw_+\right) (\xi)\left[
				\int_{\left\lbrace u(\xi)< u(\eta)\right\rbrace\setminus B_r }
				g\left( \frac{|u(\xi)-u(\eta)|}{|\eta^{-1}\circ\xi|^{s}_{\mathbb{H}^n}}\right)
				\frac{u(\xi)-u(\eta)}{|u(\xi)-u(\eta)|}
				\frac{d\eta}{|\eta^{-1}\circ\xi|^{Q+s}_{\mathbb{H}^n}}
				\right]
				d\xi\nonumber\\
				&\ge \int_{B_\rho}w_+(\xi)\left[
				\int_{\left\lbrace u(\xi)\ge u(\eta)\right\rbrace\setminus B_r }
				g\left( \frac{|u(\xi)-u(\eta)|}{|\eta^{-1}\circ\xi|^{s}_{\mathbb{H}^n}}\right)
				\frac{d\eta}{|\eta^{-1}\circ\xi|^{Q+s}_{\mathbb{H}^n}}
				\right]
				d\xi\nonumber\\
				&\quad -\int_{B_{\frac{r+\rho}{2}}}w_+(\xi)\left[
				\int_{\left\lbrace u(\xi)< u(\eta)\right\rbrace\setminus B_r }
				g\left( \frac{|u(\xi)-u(\eta)|}{|\eta^{-1}\circ\xi|^{s}_{\mathbb{H}^n}}\right)
				\frac{d\eta}{|\eta^{-1}\circ\xi|^{Q+s}_{\mathbb{H}^n}}
				\right]
				d\xi\nonumber\\
				&=:I_{21}-I_{22}.
			\end{align}
		For $I_{21}$, we find that
			\begin{align}
				\label{1.8}
				I_{21}&\ge \int_{B_\rho}w_+(\xi)\int_{A^-(k)\setminus B_r}
				g\left( \frac{w_+(\xi)+w_-(\eta)}{|\eta^{-1}\circ\xi|^{s}_{\mathbb{H}^n}}\right)
				\frac{d\eta}{|\eta^{-1}\circ\xi|^{Q+s}_{\mathbb{H}^n}}\,d\xi\nonumber\\
				&\ge \int_{B_\rho}w_+(\xi)\int_{\mathbb{H}^n\setminus B_r}
				g\left( \frac{w_-(\eta)}{|\eta^{-1}\circ\xi|^{s}_{\mathbb{H}^n}}\right)
				\frac{d\eta}{|\eta^{-1}\circ\xi|^{Q+s}_{\mathbb{H}^n}}\,d\xi.
			\end{align}
		On the other hand, if $\xi\in B_{\frac{r+\rho}{2}}$ and $\eta\in \mathbb{H}^n\setminus B_r$, then
		\begin{equation*}
\begin{split}
			|\xi_0^{-1}\circ\eta|_{\mathbb{H}^n}&\le \left( 1+\frac{|\xi_0^{-1}\circ\eta|_{\mathbb{H}^n}}{|\eta^{-1}\circ\xi|_{\mathbb{H}^n}}\right) |\eta^{-1}\circ\xi|_{\mathbb{H}^n}\\
			&\le \left( 1+\frac{(\rho+r)/2}{(r-\rho)/2}\right) |\eta^{-1}\circ\xi|_{\mathbb{H}^n}
			\le \frac{2r}{r-\rho}|\eta^{-1}\circ\xi|_{\mathbb{H}^n}.
\end{split}
		\end{equation*}
		Thereby,
			\begin{align}
				\label{1.9}
				I_{22}&\le \int_{B_{\frac{r+\rho}{2}}}w_+(\xi)\left[
				\int_{\mathbb{H}^n\setminus B_r}g\left(
				\left( \frac{2r}{r-\rho}\right)^s\frac{w_+(\eta)}{|\xi_0^{-1}\circ\eta|^s_{\mathbb{H}^n}}
				\right)
				\left( \frac{2r}{r-\rho}\right)^{Q+s} \frac{d\eta}{|\xi_0^{-1}\circ\eta|^{Q+s}_{\mathbb{H}^n}}
				\right] d\xi\nonumber\\
				&\le \frac{q}{p}\left( \frac{2r}{r-\rho}\right)^{Q+sq} \int_{B_{\frac{r+\rho}{2}}}w_+(\xi)
				\int_{\mathbb{H}^n\setminus B_r}g\left(
				\frac{w_+(\eta)}{|\xi_0^{-1}\circ\eta|^s_{\mathbb{H}^n}}
				\right)
				\frac{d\eta}{|\xi_0^{-1}\circ\eta|^{Q+s}_{\mathbb{H}^n}}\,
				d\xi\nonumber\\
				&\le C\left( \frac{r}{r-\rho}\right)^{Q+sq}||w_+||_{L^1(B_r)}\mathrm{Tail}(w_+;\xi_0,r),
			\end{align}
		where we utilized \eqref{02} and \eqref{1-1}, and the positive constant $C$ depends only open $p$,$q$,$s$,$n$. Putting together \eqref{1.8},\eqref{1.9} and \eqref{1.7} obtains
			\begin{align}
				\label{1.10}
				I_2&\ge \int_{B_\rho}w_+(\xi)\int_{\mathbb{H}^n\setminus B_r}
				g\left( \frac{w_-(\eta)}{|\eta^{-1}\circ\xi|^{s}_{\mathbb{H}^n}}\right)
				\frac{d\eta}{|\eta^{-1}\circ\xi|^{Q+s}_{\mathbb{H}^n}}\,d\xi\nonumber\\
				&\quad-C\left( \frac{r}{r-\rho}\right)^{Q+sq}||w_+||_{L^1(B_r)}\mathrm{Tail}(w_+;\xi_0,r).
			\end{align}
		Finally, we combine \eqref{1.6},\eqref{1.10} with \eqref{1.1} to infer the desired result.
	\end{proof}
	
Now we show an integral form of Sobolev-Poincar\'{e} type inequality for functions in the fractional Orlicz-Sobolev space $H\mathbb W^{s,G}(B_r)$ in the Heisenberg group context, which is a key ingredient of the proof of regularity results.

	\begin{lemma}[Sobolev-Poincar\'{e} inequality]
		\label{B}
		Let $u\in H\mathbb W^{s,G}(B_r) (s\in (0,1))$ and $G$ be an $N$-function such that the $\Delta_2$ and $\nabla_2$ conditions 
holds with constants $\kappa$ and $\nu$. Then there is a constant $\theta=\theta(Q,s)>1$ satisfying
	    \begin{equation*}
	    	\left( \mint_{B_{r}}G^\theta\left( \frac{|u-(u)_{B_r}|}{r^s}\right) d\xi \right)^{\frac{1}{\theta}}
	    	\le C\mint_{B_{r}}\int_{B_r}G\left( \frac{|u(\xi)-u(\eta)|}{|\eta^{-1}\circ\xi|^{s}_{\mathbb{H}^n}}\right)\frac{d\xi d\eta}{|\eta^{-1}\circ\xi|^{Q}_{\mathbb{H}^n}},
	    \end{equation*}
	    where $C>0$ depends on $Q,s,\kappa$ and $\nu$.
	\end{lemma}

\begin{remark}
The proof is the same as that of \cite[Lemma 4.1]{BKO23}, except that the dimension $n$ there is substituted by $Q$ and the inequality above \cite[display (4.3)]{BKO23} is replaced with
	\begin{align*}
		|u(\xi)-(u)_{B_r}|&\le C\sum_{i=0}^{\infty}r_i^{-Q+s}\int_{B_{2r_i}(\xi)\cap B_r}h(\eta)\,d\eta\\
		&\le C\sum_{i=0}^{\infty}\sum_{j=1}^{\infty}2^{(Q-s)i}\int_{\left( B_{2r_j}(\xi)\setminus B_{2r_{j+1}}(\xi)\right) \cap B_r}r^{-Q+s}h(\eta)\,d\eta\\
		&=C\sum_{j=0}^{\infty}\sum_{i=0}^{j}2^{(Q-s)i}\int_{\left( B_{2r_j}(\xi)\setminus B_{2r_{j+1}}(\xi)\right) \cap B_r}r^{-Q+s}h(\eta)\,d\eta\\
		&\le C\sum_{j=0}^{\infty}2^{(Q-s)j}\int_{\left( B_{2r_j}(\xi)\setminus B_{2r_{j+1}}(\xi)\right) \cap B_r}r^{-Q+s}h(\eta)\,d\eta\\
		&=C2^{Q-s}\sum_{j=0}^{\infty}\int_{\left( B_{2r_j}(\xi)\setminus B_{2r_{j+1}}(\xi)\right) \cap B_r}\frac{h(\eta)}{(2^{-j+1}r)^{Q-s}}\,d\eta\\
		&\le C\sum_{j=0}^{\infty}\int_{\left( B_{2r_j}(\xi)\setminus B_{2r_{j+1}}(\xi)\right) \cap B_r}\frac{h(\eta)}{|\xi^{-1}\circ\eta|^{Q-s}_{\mathbb{H}^n}}\,d\eta\\
		&=C\int_{B_r}\frac{h(\eta)}{|\xi^{-1}\circ\eta|^{Q-s}_{\mathbb{H}^n}}\,d\eta,
		\end{align*}
	where the notations here follows that in \cite[Lemma 4.1]{BKO23} and in the second line from the bottom we utilized $|\xi^{-1}\circ\eta|_{\mathbb{H}^n}\le 2r_j=2^{-j+1}r$.
\end{remark}

In the end of this section, we complete the proof of the local boundedness of weak solutions to \eqref{main} in Theorem \ref{C}.

\medskip

\noindent\textbf{Proof of Theorem \ref{C}}. For $i\in \mathbb{N}\cup\left\lbrace 0\right\rbrace $, let
		\begin{equation*}
			r_i:=(1+2^{-i})\frac{r}{2}, B_i:=B_{r_i}(\xi_0)
		\end{equation*}
		and
		\begin{equation*}
			k_i=(1-2^{-i})k,\ \tilde{k}_i=\frac{k_i+k_{i+1}}{2},\ w_i=(u-k_i)_+,\ \tilde{w}_i=(u-\tilde{k}_i)_+
		\end{equation*}
		with $k>0$. Obviously,
		\begin{equation*}
			k_i<\tilde{k}_i<k_{i+1}\quad w_{i+1}\le\tilde{w}_i\le w_{i}.
		\end{equation*}
		Applying Proposition \ref{A} with $\rho:=r_{i+1},\ r:=r_i$ and $w_+:=\tilde{w}_i$ deduces
			\begin{align}
				\label{3.1}
				&\quad\int_{B_{i+1}}\int_{B_{i+1}}
				G\left( \frac{|\tilde{w}_i(\xi)-\tilde{w}_i(\eta)|}{|\eta^{-1}\circ\xi|^{s}_{\mathbb{H}^n}}\right)
				\frac{d\xi d\eta}{|\eta^{-1}\circ\xi|^{Q}_{\mathbb{H}^n}}\nonumber\\
				&\le C\left( \frac{r_i}{r_i-r_{i+1}}\right)^q
				\int_{B_i}G\left( \frac{\tilde{w}_i}{r_i^s}\right)\,d\xi
				+C\left( \frac{r_i}{r_i-r_{i+1}}\right)^{Q+sq}||\tilde{w}_i||_{L^1(B_i)}\mathrm{Tail}(\tilde{w}_i;\xi_0,r_i).
			\end{align}
		According to \eqref{02} and a variant on \eqref{1-1}, we analyze
		\begin{equation*}
			C2^{-(q-1)i}g\left( \frac{k}{r^s}\right)\frac{\tilde{w}_i}{r_i^s}
			\le \frac{\tilde{w}_i}{r_i^s}g\left( \frac{\tilde{k}_i-k_i}{r_i^s}\right)
			\le \frac{\tilde{w}_i}{r_i^s}g\left( \frac{\tilde{w}_i}{r_i^s}\right)
			\le qG\left( \frac{w_i}{r_i^s}\right) ,
		\end{equation*}
		and know
		\begin{equation*}
			\frac{r_i}{r_i-r_{i+1}}\le 2^{i+2}\quad\text{and}\quad \mathrm{Tail}(\tilde{w}_i;\xi_0,r_i)\le C\mathrm{Tail}(u_+;\xi_0,r/2).
		\end{equation*}
		As a result, \eqref{3.1} is evaluated as
		\begin{equation*}
			\begin{aligned}
				&\quad\mint_{B_{i+1}}\int_{B_{i+1}}{\tiny }
				G\left( \frac{|\tilde{w}_i(\xi)-\tilde{w}_i(\eta)|}{|\eta^{-1}\circ\xi|^{s}_{\mathbb{H}^n}}\right)
				\frac{d\xi d\eta}{|\eta^{-1}\circ\xi|^{Q}_{\mathbb{H}^n}}\\
				&\le C2^{qi}\mint_{B_i}G\left( \frac{w_i}{r_i^s}\right)\,d\xi
				 +C2^{(Q+sq+q)i}\mathrm{Tail}(u_+;\xi_0,r/2)\times \frac{r_i^s}{g(k/r^s)}\mint_{B_{i}}G\left( \frac{w_i}{r_i^s}\right)\,d\xi
			\end{aligned}
		\end{equation*}
		with $C=C(n,p,q,s)$.\par
		Next we are going to use Lemma \ref{B}, Sobolev-Poincar$\acute{\text{e}}$ inequality, to derive a recursive inequality. Due to Lemma \ref{B} and the last inequality,
		\begin{equation*}
			\begin{aligned}
				&\quad\left( \mint_{B_{i+1}}G^\theta\left( \frac{|\tilde{w}_i(\xi)-\tilde{w}_i(\eta)|}{r^s_{i+1}}\right)\, d\xi \right)^{\frac{1}{\theta}}\\
				&\le C2^{qi}\mint_{B_{i}}G\left( \frac{w_i}{r^s_{i}}\right)\,d\xi +C2^{(Q+sq+q)i}\mint_{B_{i}}G\left( \frac{w_i}{r_i^s}\right)\,d\xi\times \frac{r_i^s}{g(k/r^s)}\mathrm{Tail}(u_+;\xi_0,r/2)
			\end{aligned}
		\end{equation*}
		where $\theta=\theta(Q,s)>1$ and $C=C(n,p,q,s)\ge 1$. On the other hand, by Jensen's inequality, we have
		\begin{equation*}
			\begin{aligned}
				&\quad \left( \mint_{B_{i+1}}G^\theta\left( \frac{\tilde{w}_i}{r^s_{i+1}}\right) \,d\xi\right)^{\frac{1}{\theta}}\\
				&\le C\left( \mint_{B_{i+1}}G^\theta\left( \frac{|\tilde{w}_i-(\tilde{w}_i)_{B_{i+1}}|}{r^s_{i+1}}\right)\, d\xi \right)^{\frac{1}{\theta}}
				+CG\left( \frac{(\tilde{w}_i)_{B_{i+1}}}{r_{i+1}^s}\right) \\
				&\le C\left( \mint_{B_{i+1}}G^\theta\left( \frac{|\tilde{w}_i-(\tilde{w}_i)_{B_{i+1}}|}{r^s_{i+1}}\right)\, d\xi \right)^{\frac{1}{\theta}}
				+C\mint_{B_{i+1}}G\left( \frac{\tilde{w}_i}{r^s_{i+1}}\right)\,d\xi.
			\end{aligned}
		\end{equation*}
		Observe that
		\begin{equation*}
			\begin{aligned}
				C^{-1}(\theta,q)2^{-q\theta i}G^{\theta-1}\left( \frac{k}{r^s}\right) G\left( \frac{w_{i+1}}{r^s_{i+1}}\right)
				&\le G\left( 2^{-i-2}k/r^s\right)^{\theta-1}G\left( \frac{w_{i+1}}{r^s_{i+1}}\right)\\
				&=G^{\theta-1}\left( \frac{k_{i+1}-\tilde{k}_i}{r^s_{i+1}}\right)G\left( \frac{w_{i+1}}{r^s_{i+1}}\right)\\
				&\le G^{\theta-1}\left( \frac{\tilde{w}_i}{r^s_{i+1}}\right)G\left( \frac{w_{i+1}}{r^s_{i+1}}\right)\\
				&\le G^\theta\left( \frac{\tilde{w}_i}{r^s_{i+1}}\right)
			\end{aligned}
		\end{equation*}
		from the monotonicity of $G$. This time it follows from the three displays above that
		\begin{equation*}
			\begin{aligned}
				&\quad\  G^{\frac{\theta-1}{\theta}}\left( \frac{k}{r^s}\right)
				\left( \mint_{B_{i+1}}G\left( \frac{w_{i+1}}{r^s_{i+1}}\right) d\xi\right)^{\frac{1}{\theta}}\\
				&\le C(\theta,q)2^{qi}  \left[
				C\left( \mint_{B_{i+1}}G^{\theta}\left( \frac{|w_{i+1}-(w_{i+1})_{B_{i+1}}|}{r^s_{i+1}}\right) d\xi\right)^{\frac{1}{\theta}}
				+C \mint_{B_{i+1}}G\left( \frac{w_{i+1}}{r^s_{i+1}}\right)d\xi
				\right] \\
				&\le C2^{2qi}\mint_{B_{i}}G\left( \frac{w_i}{r^s_i}\right)d\xi+C2^{(Q+sq+2q)i}\frac{r^s_i}{g(k/r^s)}\mathrm{Tail}(u_+;\xi_0,r/2)\mint_{B_{i}} G\left( \frac{w_i}{r^s_i}\right)d\xi\\
				&\le C2^{(Q+sq+2q)i}\left( 1+\frac{r^s\mathrm{Tail}(u_+;\xi_0,r/2)}{g(k/r^s)}\right)\mint_{B_{i}}G\left( \frac{w_i}{r^s_i}\right)d\xi.
			\end{aligned}
		\end{equation*}
		We set
		\begin{equation*}
			Y_i=\frac{\mint_{B_{i}}G(w_i/r^s_i)\, d\xi}{G(k/r^s)} ,
		\end{equation*}
		and obtain
		\begin{equation*}
			Y_{i+1}\le Cb^i\left( 1+\frac{r^s\mathrm{Tail}(u_+;\xi_0,r/2)}{g(k/r^s)}\right)^\theta Y_i^\theta
		\end{equation*}
		with $C=C(n,p,q,s)\ge 1$ and $b=2^{\theta(Q+sq+2q)}$. Picking first $k$ so large that
		\begin{equation*}
			k\ge r^sg^{-1}\left( \delta(r/2)^s\mathrm{Tail}(u_+;\xi_0,r/2)\right)
		\end{equation*}
		with $\delta\in \left( 0,1\right] $, we can discover
		\begin{equation*}
			Y_{i+1}\le Cb^i\left( 1+\frac{1}{\delta}\right)^\theta Y_i^\theta\le (\delta^{-\theta}C)b^iY_i^\theta.
		\end{equation*}
		If
		\begin{equation}
			Y_0=\frac{\mint_{B_r}G(u_+/r^s)\, d\xi}{G(k/r^s)} \le \left( \delta^{-\theta}C\right)^{\frac{1}{1-\theta}}b^{-\frac{1}{(\theta-1)^2}},
			\label{3.2}
		\end{equation}
		then by the iteration lemma (\cite[Lemma 7.1]{Giu}) there holds that $Y_i\rightarrow0$ as $i\rightarrow\infty$. At this moment, we could select
		\begin{equation*}
			k=r^sG^{-1}\left( \delta^{\frac{\theta}{1-\theta}}C^{\frac{1}{\theta-1}}b^{\frac{1}{(\theta-1)^2}}\mint_{B_r}G\left( \frac{u_+}{r^s}\right)d\xi \right)
			+r^sg^{-1}\left( \delta(r/2)^s\mathrm{Tail}(u_+;\xi_0,r/2)\right),
		\end{equation*}
		so that \eqref{3.2} is true. Eventually, the limit $\lim\limits_{i\rightarrow\infty}Y_i=0$ implies the desired result.

	\section{H\"{o}lder regularity}
    \label{sec-3}
	In this part, we are ready to justify that the weak solution, $u\in H\mathbb W^{s,G}(\Omega)\cap L^g_s(\mathbb{H}^n)$ to Eq. \eqref{main}, is locally H\"{o}lder continuous in $\Omega$. For this aim, we first establish some crucial results; see Lemmas \ref{D}--\ref{F}.\par
	We now suppose weak solutions are locally bounded and introduce some numbers $\nu^\pm$ and $\omega$ fulfilling
	\begin{equation*}
		\nu^-\le \inf_{B_R(\xi_0)}u,\quad \nu^+\ge \sup_{B_R(\xi_0)}u,\quad \omega\ge \nu^+-\nu^-,
	\end{equation*}
	where $B_R(\xi_0)\subset\subset\Omega$.

	\begin{lemma}
		\label{D}
		Let $u\in H\mathbb W^{s,G}(\Omega)\cap L^g_s(\mathbb{H}^n)$ be a locally bounded weak sub(super)-solution to Eq. \eqref{main}. Assume that for some $\beta$ and $\gamma$ in $(0,1)$, there holds that
		\begin{equation}
			\left| \left\lbrace \pm (\nu^\pm-u)\ge \gamma \omega\right\rbrace\cap B_{2r}(\xi_0) \right|\ge \beta\left| B_{2r}\right|.
			\label{4.0}
		\end{equation}
		Then for some $\sigma\in \left( 0,\frac{1}{2}\right] $, we conclude that either
		\begin{equation}
			(4r)^s\mathrm{Tail}((u-\nu^\pm)_\pm;\xi_0,R)>g\left( \frac{\sigma\gamma\omega}{(4r)^s}\right)
			\label{4.01}
		\end{equation}
		or
		\begin{equation*}
			\left| B_{2r}(\xi_0)\cap \left\lbrace \pm(\nu^\pm-u)\le \frac{\sigma\gamma\omega}{2}\right\rbrace \right|\le \frac{C\sigma^{p-1}}{\beta}\left| B_{2r}\right|
		\end{equation*}
		with $C>0$ depending only on $n,p,q$ and $s$, provided that $B_{4r}(\xi_0)\subset B_R(\xi_0)$.
	\end{lemma}
	\begin{proof}
		We just show this statement for supersolution, as the case of subsolution is similar. Let
		\begin{equation*}
			\bar{u}:=u-\nu^-\quad \text{and} \quad w:=(\bar{u}-k)
		\end{equation*}
		with $k=\sigma\gamma\omega$. Use Proposition \ref{A} to have
		\begin{equation}
			\begin{aligned}
				&\quad\ \int_{B_{2r}}w_-(\xi)\left[ \int_{B_{2r}}g\left( \frac{w_+(\eta)}{|\eta^{-1}\circ\xi|^{s}_{\mathbb{H}^n}}\right) \frac{d\eta}{|\eta^{-1}\circ\xi|^{Q+s}_{\mathbb{H}^n}}\right] d\xi\\
				&\le C\left( \frac{4r}{4r-2r}\right)^q\int_{B_{4r}}G\left( \frac{w_-}{(4r)^s}\right)d\xi+C\left( \frac{4r}{4r-2r}\right)^{Q+sq}||w_-||_{L^1(B_{4r})}\mathrm{Tail}(w_-;\xi_0,4r)\\
				&\le CG\left( \frac{k}{(4r)^s}\right)\left| B_{4r}\right|+Ck\left| B_{4r}\right|\mathrm{Tail}(w_-;\xi_0,4r),
			\end{aligned}\label{4.1}
		\end{equation}
		where we notice that $u-\nu^-\ge 0$ in $B_{4r}$. Next we deal with the nonlocal tail in the last line. 
		By means of the properties of $g$ and $u-\nu^-\ge 0$ in $B_R$,
			\begin{align}
				\label{4.2}
				 \mathrm{Tail}(w_-;\xi_0,4r)
				&\le \int_{\mathbb{H}^n\setminus B_R}g\left( \frac{\bar{u}_-+k}{|\xi_0^{-1}\circ\xi|^{s}_{\mathbb{H}^n}}\right) \frac{d\xi}{|\xi_0^{-1}\circ\xi|^{Q+s}_{\mathbb{H}^n}}\nonumber\\
				&\quad+\int_{B_R\setminus B_{4r}}g\left( \frac{k}{|\xi_0^{-1}\circ\xi|^{s}_{\mathbb{H}^n}}\right) \frac{d\xi}{|\xi_0^{-1}\circ\xi|^{Q+s}_{\mathbb{H}^n}}\nonumber\\
				&\le C(p,q)\int_{\mathbb{H}^n\setminus B_R}g\left( \frac{\bar{u}_-}{|\xi_0^{-1}\circ\xi|^{s}_{\mathbb{H}^n}}\right) \frac{d\xi}{|\xi_0^{-1}\circ\xi|^{Q+s}_{\mathbb{H}^n}}\nonumber\\
				&\quad+C(p,q)\int_{\mathbb{H}^n\setminus B_{4r}}g\left( \frac{k}{|\xi_0^{-1}\circ\xi|^{s}_{\mathbb{H}^n}}\right) \frac{d\xi}{|\xi_0^{-1}\circ\xi|^{Q+s}_{\mathbb{H}^n}}\nonumber\\
				&\le C\mathrm{Tail}(\bar{u}_-;\xi_0,R)+C(4r)^{(p-1)s}g\left( \frac{k}{(4r)^s}\right) \int_{\mathbb{H}^n\setminus B_{4r}}\frac{d\xi}{|\xi_0^{-1}\circ\xi|^{Q+sp}_{\mathbb{H}^n}}\nonumber\\
				&\le C\mathrm{Tail}(\bar{u}_-;\xi_0,R)+C\frac{1}{4r}g\left( \frac{k}{(4r)^s}\right)
			\end{align}
		with $C=C(Q,p,q,s)>0$. In view of \eqref{4.1} and \eqref{4.2}, we find
			\begin{align}
				\label{4.3}
				&\quad\ \int_{B_{2r}}w_-(\xi)\left[ \int_{B_{2r}}g\left( \frac{w_+(\eta)}{|\eta^{-1}\circ\xi|^{s}_{\mathbb{H}^n}}\right) \frac{d\eta}{|\eta^{-1}\circ\xi|^{Q+s}_{\mathbb{H}^n}}\right] d\xi\nonumber\\
				&\le CG\left( \frac{k}{(4r)^s}\right)\left| B_{4r}\right|+Ck\mathrm{Tail}(\bar{u}_-;\xi_0,R)\left| B_{4r}\right| \nonumber\\
				&\le CG\left( \frac{k}{(4r)^s}\right)\left| B_{4r}\right|\nonumber\\
&\le C\sigma^pG\left( \frac{\gamma\omega}{(4r)^s}\right) \left| B_{4r}\right| ,
			\end{align}
		where the converse of the condition \eqref{4.01} was applied.\par
		On the other hand, the integral at the left-hand side can be estimated as blow,
			\begin{align}
				\label{4.4}
				&\ \quad\int_{B_{2r}}(\bar{u}-k)_-\left[ \int_{B_{2r}}g\left( \frac{(\bar{u}-k)_+}{|\eta^{-1}\circ\xi|^{s}_{\mathbb{H}^n}}\right) \frac{d\eta}{|\eta^{-1}\circ\xi|^{Q+s}_{\mathbb{H}^n}}\right] d\xi\nonumber\\
				&\ge \int_{B_{2r}\cap \left\lbrace \bar{u}\le \frac{k}{2}\right\rbrace }(\bar{u}-k)_-
				\left[ \int_{B_{2r}\cap \left\lbrace \bar{u}\ge \gamma\omega\right\rbrace}g\left( \frac{(\bar{u}-k)_+}{|\eta^{-1}\circ\xi|^{s}_{\mathbb{H}^n}}\right) \frac{d\eta}{|\eta^{-1}\circ\xi|^{Q+s}_{\mathbb{H}^n}}\right] d\xi\nonumber\\
				&\ge \frac{k}{2}\frac{1}{(4r)^{Q+s}}g\left( \frac{\gamma\omega-k}{(4r)^s}\right)
				\left| B_{2r}\cap \left\lbrace \bar{u}\le k/2\right\rbrace\right|
				\left| B_{2r}\cap \left\lbrace \bar{u}\ge \gamma\omega\right\rbrace\right| \nonumber\\
				&\ge \frac{k}{2}\frac{\beta}{(4r)^{Q+s}}g\left( \frac{\gamma\omega(1-\sigma)}{(4r)^s}\right)
				\left| B_{2r}\right| \left| B_{2r}\cap \left\lbrace \bar{u}\le k/2\right\rbrace\right|\nonumber\\
				&\ge \frac{\sigma\gamma\omega}{C}\frac{\beta}{(4r)^{s}}g\left( \frac{\gamma\omega}{2(4r)^s}\right)\left| B_{2r}\cap \left\lbrace \bar{u}\le k/2\right\rbrace\right|\nonumber\\
				&\ge \frac{\sigma\beta}{C}G\left( \frac{\gamma\omega}{(4r)^s}\right) \left| B_{2r}\cap \left\lbrace \bar{u}\le k/2\right\rbrace\right|,
			\end{align}
		where in the third line we need note $|\eta^{-1}\circ\xi|_{\mathbb{H}^n}\le 4r$ and the condition \eqref{4.0} is used in the penultimate inequality. Therefore, it yields from \eqref{4.3} and \eqref{4.4} that
		\begin{equation*}
			\left| B_{2r}\cap \left\lbrace u-\nu^-\le\frac{\sigma\gamma\omega}{2}\right\rbrace \right|\le \frac{\sigma^{p-1}C}{\beta}\left| B_{2r}\right|
		\end{equation*}
		for the positive constant $C=C(n,p,q,s)$.
	\end{proof}

	The next result concerns a De Giorgi-type lemma.

	\begin{lemma}
		\label{E}
		Let $u\in H\mathbb W^{s,G}(\Omega)\cap L^g_s(\mathbb{H}^n)$ be a locally bounded weak sub(super)-solution to \eqref{main}. Then there is a constant $\mu\in (0,1)$, depending only on $n,p,q,s$, such that if
		\begin{equation}
			\left| B_{2r}(\xi_0)\cap \left\lbrace \pm(\nu^\pm-u)\le \gamma\omega\right\rbrace \right|\le \mu \left| B_{2r}\right|
			\label{5.1}
		\end{equation}
		with some $\gamma\in (0,1)$, then either
		\begin{equation}
			r^s\mathrm{Tail}((u-\nu^\pm)_\pm;\xi_0,R)>g\left( \frac{\gamma\omega}{r^s}\right)
			\label{5.2}
		\end{equation}
		or
		\begin{equation}
			\pm(\nu^\pm-u)\ge \frac{1}{2}\gamma\omega\qquad in\ B_r(\xi_0),
			\label{5.3}
		\end{equation}
		whenever $B_{2r}(\xi_0)\subset B_R(\xi_0)$.
	\end{lemma}

	\begin{proof}
		We just prove this assertion for supersolution. Suppose \eqref{5.2} is not true, i.e.,
		\begin{equation}
			r^s\mathrm{Tail}((u-\nu^-)_-;\xi_0,R)\le g\left( \frac{\gamma\omega}{r^s}\right).
		    \label{5.4}
		\end{equation}
		We now justify \eqref{5.3}. For $i\in \mathbb{N}\cup \left\lbrace 0\right\rbrace $, define
		\begin{equation*}
			r_i=r(1+2^{-i}),\ B_i=B_{r_i}(\xi_0),\ k_i=\frac{k}{2}(1+2^{-i})
		\end{equation*}
		with $k=\gamma\omega$, and
		\begin{equation*}
			w_i=(\bar{u}-k_i)_-:=(u-\nu^--k_i)_-.
		\end{equation*}
		Through exploiting Proposition \ref{A} with $w_-:=w_i,r:=r_i$ and $\rho:=r_{i+1}$, we have
			\begin{align}
				\label{5.5}
				&\quad\ \int_{B_{i+1}}\int_{B_{i+1}}
				G\left( \frac{|w_i(\xi)-w_i(\eta)|}{|\eta^{-1}\circ\xi|^{s}_{\mathbb{H}^n}}\right) \frac{d\xi d\eta}{|\eta^{-1}\circ\xi|^{Q}_{\mathbb{H}^n}}\nonumber\\
				&\le C\left( \frac{r_i}{r_i-r_{i+1}}\right)^q\int_{B_{i}}G\left( \frac{w_i}{r^s_i}\right)d\xi
				+C\left( \frac{r_i}{r_i-r_{i+1}}\right)^{Q+sq}||w_i||_{L^1(B_i)}\mathrm{Tail}(w_i;\xi_0,r_i)\nonumber\\
				&\le C2^{qi}G\left( \frac{k}{r^s}\right)|A^-_i|+C2^{(Q+sq)i}k|A^-_i|\mathrm{Tail}(w_i;\xi_0,r_i),
			\end{align}
		where we utilized the fact $\bar{u}\ge 0$ in $B_i$, and $A^-_i$ means
		\begin{equation*}
			A^-_i:=B_i\cap \left\lbrace \bar{u}\le k_i\right\rbrace.
		\end{equation*}
		Following the computation of \eqref{4.2}, we easily get
		\begin{equation}
			\mathrm{Tail}(w_i;\xi_0,r_i)\le C\mathrm{Tail}(\bar{u}_-;\xi_0,R)+\frac{C}{r_i^s}g\left( \frac{k_i}{r^s_i}\right).
			\label{5.6}
		\end{equation}
		By virtue of \eqref{5.4}, \eqref{5.5} and \eqref{5.6},
		\begin{equation}
				\int_{B_{i+1}}\int_{B_{i+1}}
				G\left( \frac{|w_i(\xi)-w_i(\eta)|}{|\eta^{-1}\circ\xi|^{s}_{\mathbb{H}^n}}\right) \frac{d\xi d\eta}{|\eta^{-1}\circ\xi|^{Q}_{\mathbb{H}^n}}
				\le C2^{(Q+sq+q)i}G\left( \frac{k}{r^s}\right)|A^-_i| .
				\label{5.7}
		\end{equation}\par
		On the other hand, with the help of Lemma \ref{B}, Jensen's inequality and \eqref{5.7}, there holds
		\begin{equation*}
			\begin{aligned}
				&\quad\ \left(\mint_{B_{i+1}}G^\theta\left( \frac{w_i}{r^s_{i+1}}\right)d\xi  \right) ^{\frac{1}{\theta}}\\
				&\le C \left(\mint_{B_{i+1}}G^{\theta}\left( \frac{|w_i-(w_i)_{B_{i+1}}|}{r^s_{i+1}}\right) d\xi  \right) ^{\frac{1}{\theta}}
				+C\mint_{B_{i+1}}G\left( \frac{w_i}{r^s_{i+1}}\right)d\xi\\
				&\le C \mint_{B_{i+1}}\int_{B_{i+1}}G\left( \frac{|w_i(\xi)-w_i(\eta)|}{|\eta^{-1}\circ\xi|^{s}_{\mathbb{H}^n}}\right) \frac{d\xi d\eta}{|\eta^{-1}\circ\xi|^{Q}_{\mathbb{H}^n}}
				+C\mint_{B_{i+1}}G\left( \frac{w_i}{r^s_{i+1}}\right)d\xi\\
				&\le C2^{(Q+sq+q)i}G\left( \frac{k}{r^s}\right)\frac{|A^-_i|}{|B_i|} .
			\end{aligned}
		\end{equation*}
		We proceed to estimate
		\begin{equation*}
			\begin{aligned}
				\left(\mint_{B_{i+1}}G^\theta\left( \frac{w_i}{r^s_{i+1}}\right)d\xi  \right) ^{\frac{1}{\theta}}
				&\ge \frac{1}{|B_{i+1}|^{\theta-1}}\left(\int_{A^-_{i+1}}G^\theta\left( \frac{w_i}{r^s_{i+1}}\right)d\xi  \right) ^{\frac{1}{\theta}}\\
				&\ge |B_{i+1}|^{-\frac{1}{\theta}}\left(\int_{A^-_{i+1}}G^\theta\left( \frac{k_i-k_{i+1}}{r^s_{i+1}}\right)d\xi  \right) ^{\frac{1}{\theta}}\\
				&\ge C2^{-qi}G\left( \frac{k}{r^s}\right) \left( \frac{|A^-_{i+1}|}{|B_{i+1}|}\right) ^{\frac{1}{\theta}},
			\end{aligned}
		\end{equation*}
		where the constant $\theta=\theta(Q,s)>1$ comes from Lemma \ref{B}. Merging the previous two displays arrives at
		\begin{equation*}
			\frac{|A^-_{i+1}|}{|B_{i+1}|}\le C2^{\theta(Q+sq+2q)i}\left( \frac{|A^-_i|}{|B_i|}\right)^\theta.
		\end{equation*}
		Set
$$
Y_i=\frac{|A^-_i|}{|B_i|}.
$$
We employ the iteration lemma (\cite[Lemma 7.1]{Giu}) to infer $Y_i\rightarrow0$ as $i\rightarrow\infty$, provided
		\begin{equation*}
			Y_0\le C^{-\frac{1}{\theta-1}}2^{-\theta(Q+sq+2q)\frac{1}{(\theta-1)^2}}.
		\end{equation*}
		This can be assured by taking $\mu=C^{-\frac{1}{\theta-1}}2^{-\theta(Q+sq+2q)\frac{1}{(\theta-1)^2}}$ in \eqref{5.1}. Hence we can conclude $\bar{u}\ge \frac{k}{2},\ \text{i.e.,}\ u-\nu^-\ge \frac{\gamma\omega}{2}\ \text{in}\ B_r$.
	\end{proof}

	As a consequence of Lemmas \ref{D} and \ref{E}, we derive the following growth lemma.

	\begin{corollary}
		\label{F}
		Let $u\in H\mathbb W^{s,G}(\Omega)\cap L^g_s(\mathbb{H}^n)$ be a locally bouneded weak sub(super)-solution to Eq. \eqref{main}. If $B_{2r}(\xi_0)\subset B_R(\xi_0)$ and, for some $\beta,\gamma\in (0,1)$,
		\begin{equation}
			\left|
			B_r(\xi_0)\cap \left\lbrace \pm (\nu^\pm-u)\ge \gamma\omega\right\rbrace
			\right| \ge \beta |B_r|,
			\label{6.1}
		\end{equation}
		then there exists a constant $\tau\in\left( 0,\frac{1}{2}\right] $, depending only upon $n,p,q,s$ and $\beta$, such that either
		\begin{equation*}
			r^s\mathrm{Tail}((u-\nu^\pm)_\pm;\xi_0,R)>g\left( \frac{\tau\gamma\omega}{r^s}\right)
		\end{equation*}
		or
		\begin{equation*}
			\pm(\nu^\pm-u)\ge \tau\gamma\omega\qquad\text{in}\ B_{r}(\xi_0).
		\end{equation*}
	\end{corollary}
	\begin{proof}
		By \eqref{6.1}, we can see
		\begin{equation*}
			\left|
			\left\lbrace \pm(\nu^\pm-u)\ge \gamma\omega\right\rbrace \cap B_{2r}(\xi_0)
			\right|
			\ge \frac{\beta}{2^Q}|B_{2r}|.
		\end{equation*}
		According to Lemmas \ref{D} and \ref{E}, we could choose $\sigma\in \left( 0,\frac{1}{2}\right] $ so small that $\frac{C\sigma^{p-1}}{\beta}$ is not larger than $\mu$. Then this conclusion follows readily after simple arrangement.
	\end{proof}
	
	In what follows, we are going to prove the H\"{o}lder continuity of solutions, Theorem \ref{thm2}, by induction. The key points of proof are to choose suitably decreasing sequences of radii $\{r_j\}$ and of bound on oscillation $\{\omega_j\}$, and moreover to control precisely the nonlocal tail at each step in the process of iteration.
	
	\subsection{Step 1 of induction}
	Fix $\xi_0\in \Omega$ and $B_R=B_R(\xi_0)\subset\subset \Omega$. Set
	\begin{equation*}
		\nu^+=\sup_{B_R}u,\ \nu^-=\inf_{B_R}u\quad \text{and}\quad \omega=2\sup_{B_R}|u|+R^sg^{-1}\left( R^s\mathrm{Tail}(u;\xi_0,R)\right).
	\end{equation*}
	Obviously,
	\begin{equation}
		\label{2.1-1}
		\osc\limits_{B_R}u=\nu^+-\nu^-\le \omega,
	\end{equation}
	which is the start towards H$\ddot{\text{o}}$lder continuity of $u$. Let $\lambda\in \left( 0,\frac{1}{4}\right) $ be a number to be fixed later. Suppose with no loss of generality $\nu^+-\nu^->\frac{\omega}{2}$. Then one of the following scenarios
	\begin{equation*}
		\left\lbrace
		\begin{aligned}
			\left| \left\lbrace u-\nu^->\frac{1}{4}\omega\right\rbrace \cap B_{\lambda R}(\xi_0)\right|\ge \frac{1}{2}\left| B_{\lambda R}\right|  \\
			\left| \left\lbrace \nu^+-u>\frac{1}{4}\omega\right\rbrace \cap B_{\lambda R}(\xi_0)\right|\ge \frac{1}{2}\left| B_{\lambda R}\right|
		\end{aligned}
		\right.
	\end{equation*}
	must hold true. Assume the first alternative holds, since the other is similar. We this time apply Corollary \ref{F} with $\beta=\frac{1}{2},\gamma=\frac{1}{4}$ and $r=\lambda R$ to discover that either
	\begin{equation}
		\label{2.1-2}
		(\lambda R)^s\mathrm{Tail}((u-\nu^-)_-;\xi_0,R)>g\left( \frac{\tau\omega}{(\lambda R)^s}\right)
	\end{equation}
	or
	\begin{equation*}
		u-\nu^-\ge \tau\omega\quad \text{in }B_{\lambda R}(\xi_0),
	\end{equation*}
	which together with \eqref{2.1-1} indicates
	\begin{equation*}
		\osc\limits_{B_{\lambda R}}u=\sup_{B_{\lambda R}}(u-\nu^-)-\inf_{B_{\lambda R}}(u-\nu^-)\le (1-\tau)\omega,
	\end{equation*}
	where $\tau\in \left( 0,\frac{1}{2}\right) $ depends only on $n,p,q,s$.\par
	Now we will take a suitable $\lambda$ to guarantee \eqref{2.1-2} does not occur. Via the definition of $\omega$, the nonlocal tail is estimated as
	\begin{equation*}
		\begin{aligned}
			&\quad R^s\mathrm{Tail}((u-\nu^-)_-;\xi_0,R)\\
			&\le CR^s\int_{\mathbb{H}^n\setminus B_R}\left[ g\left( \frac{u_-}{|\xi_0^{-1}\circ\xi|^{s}_{\mathbb{H}^n}}\right)
			+g\left( \frac{\nu^-}{|\xi_0^{-1}\circ\xi|^{s}_{\mathbb{H}^n}}\right)
			\right]
			\frac{d\xi}{|\xi_0^{-1}\circ\xi|^{Q+s}_{\mathbb{H}^n}}\\
			&\le CR^s\mathrm{Tail}(u_-;\xi_0,R)+CR^sg\left( \frac{\omega}{R^s}\right)\int_{\mathbb{H}^n\setminus B_R}\frac{R^{s(p-1)}}{|\xi_0^{-1}\circ\xi|^{Q+sp}_{\mathbb{H}^n}}d\xi\\
			&\le Cg\left( \frac{\omega}{R^s}\right),
		\end{aligned}
	\end{equation*}
	where we used
	\begin{equation*}
		g(a+b)\le C(p,q)\left( g(a)+g(b)\right)\quad \text{and}\quad g(\tau a)\le  c(p,q)\tau^{p-1}g(a)\quad \text{for } \tau\le 1.
	\end{equation*}
	In order to ensure \eqref{2.1-2} does not happen, we select $\lambda$ to satisfy
	\begin{equation*}
		C\lambda^sg\left( \frac{\omega}{R^s}\right) \le g\left( \frac{\tau\omega}{(\lambda R)^s}\right)\Rightarrow g^{-1}\left( C\lambda^sg\left( \frac{\omega}{R^s}\right) \right)\le \frac{\tau\omega}{(\lambda R)^s},
	\end{equation*}
	namely, by letting $C\lambda^s\le 1$,
	\begin{equation}
		\label{2.1-2-2}
		\frac{\tau\omega}{(\lambda R)^s}\ge \left( \frac{q}{pC\lambda^s}\right)^{-\frac{1}{p-1}}\frac{\omega}{R^s}\Rightarrow \lambda\le \left( \frac{\tau}{C}\right)^{\frac{p-1}{sp}},
	\end{equation}
	where the constant $C>4$ depends on $n,p,q,s$ and we utilized
	\begin{equation*}
		g^{-1}(at)\ge \left( \frac{q}{ap}\right)^{-\frac{1}{p-1}}g^{-1}(t)\quad \text{for }0<a\le 1.
	\end{equation*}
	At this point, we have obtained the formal first step for the inductive argument, that is,
	\begin{equation*}
		\osc\limits_{B_{r_1}}u\le \omega_1:=(1-\tau)\omega,
	\end{equation*}
	with $r_1=\lambda R$ for some $\lambda\in \left( 0,(\tau/C)^{\frac{p-1}{sp}}\right] $. It is worth pointing out that we shall determine the more accurate range of $\lambda$ to meet the demands of iteration process.

	\subsection{The recursive process.}
	Suppose we have confirmed
	\begin{equation}
		\label{2.1-3}
		\osc \limits_{B_i}u\le \omega_i:=(1-\tau)^i\omega_0\quad \text{for }i=0,1,\ldots,j,
	\end{equation}
	where $\omega_0=\omega,B_i=B_{r_i},r_i=\lambda^iR$. Let also $\nu^+_i=\sup\limits_{B_i}u$ and $\nu^-_i=\inf\limits_{B_i}u$. We are ready to show the preceding oscillation estimate still holds for $i=j+1$. The procedure is analogous to Step 1. Without loss of generality, we may suppose $\nu^+_j-\nu^-_j>\frac{1}{2}\omega_j$. Then one of the following two alternatives
	\begin{equation*}
		\left\lbrace
		\begin{aligned}
			\left| \left\lbrace u-\nu^-_j>\frac{1}{4}\omega_j\right\rbrace \cap B_{\lambda r_j}(\xi_0)\right|\ge \frac{1}{2}\left| B_{\lambda r_j}\right|  \\
			\left| \left\lbrace \nu^+_j-u>\frac{1}{4}\omega_j\right\rbrace \cap B_{\lambda r_j}(\xi_0)\right|\ge \frac{1}{2}\left| B_{\lambda r_j}\right|
		\end{aligned}
		\right.
	\end{equation*}
	must be true. Let us assume the first case holds. Exploiting Corollary \ref{F} with $\beta=\frac{1}{2},\gamma=\frac{1}{4}$ and $r=\lambda r_j$, we obtain that either
	\begin{equation}
		\label{2.1-4}
		(\lambda r_j)^s\mathrm{Tail}((u-\nu^-_j)_-;\xi_0,r_j)>g\left( \frac{\tau\omega_j}{(\lambda r_j)^s}\right)
	\end{equation}
	or
	\begin{equation*}
		u-\nu^-_j\ge \tau\omega_j\quad \text{in } B_{j+1}(\xi_0)=B_{\lambda r_j}(\xi_0)
	\end{equation*}
	with the small constant $\tau\in \left(0,\frac{1}{2}\right) $ depending only on $n,p,q,s$. This along with \eqref{2.1-3} implies
	\begin{equation*}
		\osc\limits_{B_{j+1}}u=\sup\limits_{B_{j+1}}(u-\nu^-_j)-\inf\limits_{B_{j+1}}(u-\nu^-_j)\le (1-\tau)\omega_j=\omega_{j+1}.
	\end{equation*}\par
	At this moment, our goal is to enforce that \eqref{2.1-4} can not occur through choosing properly $\lambda\in \left( 0,\frac{1}{4}\right) $. To this end, we evaluate
	\begin{equation*}
		\begin{aligned}
			&\ \quad\mathrm{Tail}((u-\nu^-_j);\xi_0,r_j)\\
			&=\int_{\mathbb{H}^n\setminus B_R}g\left( \frac{(u-\nu^-_j)_-}{|\xi_0^{-1}\circ\xi|^{s}_{\mathbb{H}^n}}\right)\frac{d\xi}{|\xi_0^{-1}\circ\xi|^{Q+s}_{\mathbb{H}^n}}
			+\sum_{i=0}^{j-1}\int_{B_i\setminus B_{i+1}}g\left( \frac{(u-\nu^-_j)_-}{|\xi_0^{-1}\circ\xi|^{s}_{\mathbb{H}^n}}\right)\frac{d\xi}{|\xi_0^{-1}\circ\xi|^{Q+s}_{\mathbb{H}^n}}\\
			&=:I_1+I_2.
		\end{aligned}
	\end{equation*}
	For $I_1$, it follows from the definition of $\omega$ that
	\begin{equation*}
		I_1\le C\int_{\mathbb{H}^n\setminus B_R}\left[
		g\left( \frac{u_-}{|\xi_0^{-1}\circ\xi|^{s}_{\mathbb{H}^n}}\right)
		+g\left( \frac{|\nu^-_j|}{|\xi_0^{-1}\circ\xi|^{s}_{\mathbb{H}^n}}\right)
		\right] \frac{d\xi}{|\xi_0^{-1}\circ\xi|^{Q+s}_{\mathbb{H}^n}}\le \frac{C}{R^s}g\left( \frac{\omega}{R^s}\right),
	\end{equation*}
	where we note $|\nu^-_j|\le \omega$. Next let us pay attention to the integral $I_2$. For $\xi\in B_i$ with $i=0,1,\ldots,j-1$, via the inclusion relation of $B_i$,
	\begin{equation*}
		\left( u(\xi)-\nu^-_j\right)_-\le \nu^-_j-\nu^-_i\le \nu^+_j-\nu^-_i\le \nu^+_i-\nu^-_i\le \omega_i.
	\end{equation*}
	Hence,
	\begin{equation*}
		\begin{aligned}
			I_2&\le \sum_{i=0}^{j-1}\int_{B_i\setminus B_{i+1}}g\left( \frac{\omega_i}{|\xi_0^{-1}\circ\xi|^{s}_{\mathbb{H}^n}}\right) \frac{d\xi}{|\xi_0^{-1}\circ\xi|^{Q+s}_{\mathbb{H}^n}}\\
			&=\sum_{i=0}^{j-1}\int_{B_i\setminus B_{i+1}}g\left( \frac{\omega_ir^s_{i+1}}{r^s_{i+1}|\xi_0^{-1}\circ\xi|^{s}_{\mathbb{H}^n}}\right) \frac{d\xi}{|\xi_0^{-1}\circ\xi|^{Q+s}_{\mathbb{H}^n}}\\
			&\le C\sum_{i=0}^{j-1}g\left( \frac{\omega_i}{r^s_{i+1}}\right) \int_{B_i\setminus B_{i+1}}\frac{r^{s(p-1)}_{i+1}}{|\xi_0^{-1}\circ\xi|^{Q+sp}_{\mathbb{H}^n}}d\xi\\
			&\le C\sum_{i=0}^{j-1}\frac{1}{r^s_{i+1}}g\left( \frac{\omega_i}{r^s_{i+1}}\right).
		\end{aligned}
	\end{equation*}
	Finally, we arrive at
	\begin{equation*}
		r^s_j\mathrm{Tail}((u-\nu^-_j)_-;\xi_0,r_j)\le C\left( \frac{r_j}{R}\right)^sg\left( \frac{\omega}{R^s}\right)+C\sum_{i=0}^{j-1}\left( \frac{r_j}{r_{i+1}}\right)^sg\left( \frac{\omega_i}{r^s_{i+1}}\right).
	\end{equation*}
	Notice that, from the definitions of $r_i,\omega_i$,
	\begin{equation*}
		\begin{aligned}
			\left( \frac{r_j}{r_{i+1}}\right)^sg\left( \frac{\omega_i}{r^s_{i+1}}\right)
			&=\lambda^{s(j-i-1)}g\left( \frac{\omega_j}{r^s_{j+1}}\frac{r^s_{j+1}}{r^s_{i+1}}\frac{\omega_i}{\omega_j}\right)\\
			&=\lambda^{s(j-i-1)}g\left( \frac{\omega_j}{r^s_{j+1}}\lambda^{s(j-1)}(1-\tau)^{i-j} \right)\\
			&\le C(1-\tau)^{(q-1)(i-j)}\lambda^{s(j-i-1)+s(p-1)(j-1)}g\left( \frac{\omega_j}{r^s_{j+1}}\right)\\
			&=C\lambda^{-s}(1-\tau)^{(q-1)(i-j)}\lambda^{sp(j-i)}g\left( \frac{\omega_j}{r^s_{j+1}}\right)
		\end{aligned}
	\end{equation*}
	and, similarly,
	\begin{equation*}
		\left( \frac{r_j}{R}\right)^sg\left( \frac{\omega}{R^s}\right)\le C(1-\tau)^{-(q-1)j}\lambda^{sp-s+spj}g\left( \frac{\omega_j}{r^s_{j+1}}\right).
	\end{equation*}
	Then,
	\begin{equation*}
		\begin{aligned}
			r^s_j\mathrm{Tail}\left( (u-\nu^-_j)_-;\xi_0,r_j\right)&\le C\lambda^{-s}\sum_{i=0}^{j-1}\lambda^{sp(j-i)}(1-\tau)^{(1-q)(j-i)}g\left( \frac{\omega_j}{r^s_{j+1}}\right)\\
			&\le C(1-\tau)^{1-q}\lambda^{sp-s} g\left( \frac{\omega_j}{r^s_{j+1}}\right),
		\end{aligned}
	\end{equation*}
	where we can find
	\begin{equation*}
		\sum_{i=0}^{j-1}\left[ \lambda^{sp}(1-\tau)^{1-q}\right]^{j-i}=\lambda^{sp}(1-\tau)^{1-q}\frac{\left[ \lambda^{sp}(1-\tau)^{1-q}\right]^{j}-1}{\lambda^{sp}(1-\tau)^{1-q}-1}\le 2\lambda^{sp}(1-\tau)^{1-q},
	\end{equation*}
	if $\lambda^{sp}(1-\tau)^{1-q}$ is restricted by $\frac{1}{2}$, i.e.,
	\begin{equation}
		\label{2.1-5}
		\lambda\le 2^{-\frac{1}{sp}}(1-\tau)^{\frac{q-1}{sp}}.
	\end{equation}
	To contradict \eqref{2.1-4}, we need
	\begin{equation*}
		C(1-\tau)^{1-q}\lambda^{sp}g\left( \frac{\omega_j}{r^s_{j+1}}\right)\le  g\left( \frac{\tau\omega_j}{r^s_{j+1}}\right),
	\end{equation*}
	namely, by letting $C(1-\tau)^{1-q}\lambda^{sp}<1$,
	\begin{equation}
		\label{2.1-6}
		\frac{\tau\omega_j}{r^s_{j+1}}\ge \left( \frac{q}{pC(1-\tau)^{1-q}\lambda^{sp}}\right)^{-\frac{1}{p-1}}\frac{\omega_j}{r^s_{j+1}}
		\Rightarrow\lambda\le \tau^{\frac{p-1}{sp}}\left( \frac{q}{pC(1-\tau)^{1-q}}\right)^{\frac{1}{sp}}.
	\end{equation}
	Rechecking the inference process above, especially \eqref{2.1-2-2}, \eqref{2.1-5} and \eqref{2.1-6}, we select eventually $\lambda$ as
	\begin{equation}
		\label{2.1-7}
		\lambda=\frac{1}{2}\min\left\lbrace
		\left( \frac{\tau}{C}\right)^{\frac{p-1}{sp}},
		2^{-\frac{1}{sp}}(1-\tau)^{\frac{q-1}{sp}},
		\tau^{\frac{p-1}{sp}}\left( \frac{(1-\tau)^{q-1}}{C}\right)^{\frac{1}{sp}}
		\right\rbrace .
	\end{equation}
	As a mater of fact, $\lambda$ is a number in $\left( 0,\frac{1}{4}\right) $ depending only on $n,p,q,s$. As a consequence, we do infer
	\begin{equation*}
		\osc\limits_{B_{j+1}}u=\osc\limits_{B_{\lambda^{j+1}R}}u\le \omega_{j+1}
	\end{equation*}
	with $\lambda$ determined by \eqref{2.1-7}, which completes the induction and implies the H$\ddot{\text{o}}$lder continuity of $u$. Indeed, let us fix $0<r\le R$, and then there is a nonnegative integer $j$ such that
	\begin{equation*}
		\lambda^{j+1}R\le r\le \lambda^jR,
	\end{equation*}
	which results in
	\begin{equation*}
		j+1\le \text{ln}\left( \frac{r}{R}\right)^{\frac{1}{\text{ln}\lambda}}.
	\end{equation*}
	Hence,
	\begin{equation*}
		\omega_j=(1-\tau)^j\omega\le (1-\tau)^{-1}(1-\tau)^{\frac{1}{\text{ln}\lambda}\text{ln}\left( \frac{r}{R}\right) }\omega=\frac{\omega}{1-\tau}\left( \frac{r}{R}\right)^\alpha
	\end{equation*}
	with
	\begin{equation*}
		\alpha=\frac{\text{ln}(1-\tau)}{\text{ln}\lambda}.
	\end{equation*}
	In addition, we have $B_r(\xi_0)\subset B_{\lambda^jR}=B_j(\xi_0)$. Therefore, we arrive at
	\begin{equation*}
		\osc\limits_{B_r}u\le \osc\limits_{B_{j}}\le \omega_j\le \frac{\omega}{1-\tau}\left( \frac{r}{R}\right)^\alpha.
	\end{equation*}
	Now we have finished the proof of Theorem \ref{thm2}.

	\section{Harnack inequalities}
    \label{sec-4}
		The goal of this section is dedicated to establishing a full Harnack-type inequality for the weak solutions of Eq. \eqref{main}. We first need to have the forthcoming result, a variant of Corollary \ref{F}, which is deduced through following Lemma \ref{D}, Lemma \ref{E} as well as Corollary \ref{F} and letting $0=\nu^-=\inf_{B_{4R}}u$ along with $\gamma\omega:=t$ correspondingly.

		\begin{lemma}
			\label{G}
			Let $B_{4R}(\xi_0)\subset\subset\Omega$ and $u\in H\mathbb W^{s,G}(\Omega)\cap L^g_s(\mathbb{H}^n)$ be a weak solution of Eq. \eqref{main}. If
			\begin{equation*}
				u\ge 0\quad \text{in }B_{4R}(\xi_0)
			\end{equation*}
			and
			\begin{equation*}
				\left|
				B_{R}(\xi_0)\cap \left\lbrace u\ge t\right\rbrace
				\right| \ge \beta |B_R|
			\end{equation*}
			for $t>0$ and $\beta\in (0,1)$, then we can find a constant $\delta\in (0,1/8]$, which depends only on $n,p,q,s$ and $\beta$, 
satisfying that either
			\begin{equation*}
				R^s\mathrm{Tail}(u_-;\xi_0,4R)>g\left( \frac{\delta t}{R^s}\right)
			\end{equation*}
			or
			\begin{equation*}
				u\ge \delta t\quad \text{in }B_R(\xi_0).
			\end{equation*}
		\end{lemma}\par
		With the help of Lemma \ref{G} and a Krylov-Safonov covering lemma \cite[Lemma 7.2]{KS01}, we now could refine Lemma \ref{G} as follows:
		\begin{lemma}
			\label{H}
			Let $\beta\in (0,1),t>0,k\in \mathbb{N}$ and $B_{16R}(\xi_0)\subset\subset\Omega$. Assume $u\in H\mathbb W^{s,G}(\Omega)\cap L^g_s(\mathbb{H}^n)$ is a weak supersolution of Eq. \eqref{main}. When
			\begin{equation*}
				u\ge 0\quad \text{in }B_{16R}(\xi_0)
			\end{equation*}
			and
			\begin{equation*}
				\left|
				B_{R}(\xi_0)\cap \left\lbrace u\ge t\right\rbrace
				\right| \ge \beta^k |B_R|,
			\end{equation*}
			we can find a constant $\delta\in \left( 0,\frac{1}{8}\right] $, depending only upon $n,p,q,s$ and $\beta$, such that either
			\begin{equation}
				\label{3.2-1}
				R^s\mathrm{Tail}(u_-;\xi_0,16R)>g\left( \frac{\delta^k t}{R^s}\right)
			\end{equation}
			or
			\begin{equation*}
				u\ge \delta^k t\quad \text{in }B_R(\xi_0).
			\end{equation*}
		\end{lemma}\par
		The proof of this lemma is almost identical to that of \cite[Lemma 6.7]{Cozzi}, because in this process the structure of  Eq. \eqref{main} is not utilized except addressing the nonlocal tail. Corresponding to the estimate on $\mathrm{Tail}(u_-;x_0,12r)$ in \cite[Lemma 6.7]{Cozzi}, we here need to evaluate, for $\eta\in B_R(\xi_0),r\in \frac{R}{3}$ and $u\ge 0$ in $B_{12r}(\eta)$,
		\begin{equation*}
			\begin{aligned}
				(3r)^s\mathrm{Tail}(u_-;\eta,12r)&=(3r)^s\int_{\mathbb{H}^n\setminus B_{12r}(\eta)}g\left( \frac{u_-}{|\eta^{-1}\circ \xi|^s_{\mathbb{H}^n}}\right)\frac{d\xi}{|\eta^{-1}\circ \xi|^{Q+s}_{\mathbb{H}^n}}\\
				&=(3r)^s\int_{\mathbb{H}^n\setminus B_{16R}(\xi_0)}g\left( \frac{u_-}{|\eta^{-1}\circ \xi|^s_{\mathbb{H}^n}}\right)\frac{d\xi}{|\eta^{-1}\circ \xi|^{Q+s}_{\mathbb{H}^n}}\\
				&\le C(Q,p,q,s)R^s\int_{\mathbb{H}^n\setminus B_{16R}(\xi_0)}g\left( \frac{u_-}{|\xi_0^{-1}\circ \xi|^s_{\mathbb{H}^n}}\right)\frac{d\xi}{|\xi_0^{-1}\circ \xi|^{Q+s}_{\mathbb{H}^n}}\\
				&\le Cg\left( \frac{\delta^k t}{R}\right)\le g\left( \frac{\delta^k t}{3r}\right)
			\end{aligned}
		\end{equation*}
		with the use of the converse of \eqref{3.2-1}, which guarantees the reasonable application of Lemma \ref{G} in $B_{12r}(\eta)$.\par
		Based on Lemma \ref{H}, we can conclude the following weak Harnack inequality in a rather straightforward way.\par
		\begin{lemma}
			\label{I}
			Let $B_{16R}:=B_{16R}(\xi_0)\subset\Omega$ and $u\in H\mathbb W^{s,G}(\Omega)\cap L^g_s(\mathbb{H}^n)$ be a weak solution to Eq. \eqref{main}. If
			\begin{equation}
				\label{3.3-1}
				u\ge 0\quad \text{in }B_{16R},
			\end{equation}
			then there are two constants $\epsilon_0\in (0,1)$ and $C\ge 1$, both of which depend  only on $n,p,q,s$, such that
			\begin{equation*}
				\left(
				\mint_{B_R}u^{\epsilon_0}\,d\xi
				\right) ^{\frac{1}{\epsilon_0}}
				\le C\left(
				\inf_{B_R}u+R^sg^{-1}\left( R^s\mathrm{Tail}(u_-;\xi_0,R)\right)
				\right) .
			\end{equation*}
		\end{lemma}

		\begin{proof}
			Assume $u$ does not vanish identically in $B_R$. Let $\delta\in\left( 0,\frac{1}{8}\right] $ comes from Lemma \ref{H} with $\beta=\frac{1}{2}$. Define
			\begin{equation*}
				\epsilon_0:=\frac{1}{2\log_\frac{1}{2}\delta}\in (0,1).
			\end{equation*}
			We assert that, for any $t\ge 0$,
			\begin{equation}
				\label{3.3-2}
				\inf_{B_R}u+R^sg^{-1}\left( R^s\mathrm{Tail}(u_-;\xi_0,16R)\right)\ge \delta\left( \frac{|A^+(t,R)|}{|B_R|}\right)^{\frac{1}{\epsilon_0}}t.
			\end{equation}
			Via the definition of $A^+(t,R)$, we just need to verify \eqref{3.3-2} for $t\in \left[ 0,\sup_{B_R}u\right) $. Here let $\sup_{B_R}u=+\infty$ if $u$ is not bounded from above in $B_R$.\par
			Fixing arbitrarily $t\in \left[ 0,\sup\limits_{B_R}u\right) $, we can find $k=k(t)$ as the smallest integer fulfilling
			\begin{equation}
				\label{3.3-3}
				\left| A^+(t,R)\right|\ge 2^{-k}|B_R|,
			\end{equation}
			i.e., $k$ is an integer such that
			\begin{equation*}
				\log_\frac{1}{2}\frac{\left| A^+(t,R)\right|}{|B_R|}\le k<1+\log_\frac{1}{2}\frac{\left| A^+(t,R)\right|}{|B_R|}.
			\end{equation*}
			Then we get
			\begin{equation}
				\label{3.3-4}
				\delta\left( \frac{\left| A^+(t,R)\right|}{|B_R|}\right)^\frac{1}{\epsilon_0}\le \delta^k.
			\end{equation}
			If $R^sg^{-1}\left( R^s\mathrm{Tail}(u_-;\xi_0,16R)\right)>\delta^kt$, then nothing is proved for \eqref{3.3-2}. Let us consider the case $R^sg^{-1}\left( R^s\mathrm{Tail}(u_-;\xi_0,16R)\right)\le\delta^kt$, that is,
			\begin{equation*}
			 R^s\mathrm{Tail}(u_-;\xi_0,16R)\le g\left( \frac{\delta^kt}{R^s}\right) .
			\end{equation*}
			At this point, this inequality and \eqref{3.3-3} justify the application of Lemma \ref{H}, and we derive
			\begin{equation*}
				u\ge \delta^kt\quad \text{in }B_R.
			\end{equation*}\par
			From mentioned above, there holds
			\begin{equation*}
				\inf_{B_R}u+R^sg^{-1}\left( R^s\mathrm{Tail}(u_-;\xi_0,16R)\right)\ge \delta^kt,
			\end{equation*}
			and by invoking \eqref{3.3-4}, we can see that \eqref{3.3-2} is valid. We now rearrange \eqref{3.3-2} to have
			\begin{equation*}
				\frac{\left| A^+(t,R)\right|}{|B_R|}\le \left( \frac{k}{\delta t}\right)^{\epsilon_0}
			\end{equation*}
			with $k=\inf\limits_{B_R}u+R^sg^{-1}\left( R^s\mathrm{Tail}(u_-;\xi_0,16R)\right)$. Then exploiting this display and Cavalieri's principle and through some standard calculations, we could deduce
			\begin{equation*}
				\begin{aligned}
					\left(
					\mint_{B_R}u^{\epsilon_0}d\xi
					\right) ^{\frac{1}{\epsilon_0}}
					&\le C\left(
					\inf_{B_R}u+R^sg^{-1}\left( R^s\mathrm{Tail}(u_-;\xi_0,16R)\right)
					\right)\\
					&\le C\left(
					\inf_{B_R}u+R^sg^{-1}\left( R^s\mathrm{Tail}(u_-;\xi_0,R)\right)
					\right),
				\end{aligned}
			\end{equation*}
			the details of which can be found for instance in \cite[Pages 4810-4811]{Cozzi}.
		\end{proof}

		\begin{lemma}
			\label{J}
			Suppose that $u\in H\mathbb W^{s,G}(\Omega)\cap L^g_s(\mathbb{H}^n)$ is a weak solution of Eq. \eqref{main} such that $u\ge 0$ in $B_R=B_R(\xi_0)\subset\subset\Omega$. Then there holds that
			\begin{equation*}
				R^sg^{-1}\left( R^s\mathrm{Tail}(u_+;\xi_0,R)\right)
				\le C\left( \sup_{B_R}u+R^sg^{-1}\left( R^s\mathrm{Tail}(u_-;\xi_0,R)\right)\right),
			\end{equation*}
			where $C>0$ depends on $n,p,q,s$.
		\end{lemma}

		\begin{proof}
			Let $k=2\sup\limits_{B_R}u$ and $w_-=(u-k)_-$. We apply Proposition \ref{A} with $\rho=\frac{R}{2}$ and $r=R$ to get
			\begin{align}
				\label{3.4-1}
				&\ \quad\int_{B_{\frac{R}{2}}}w_-(\xi)\left[
				\int_{\mathbb{H}^n}g\left( \frac{w_+(\eta)}{|\eta^{-1}\circ \xi|^s_{\mathbb{H}^n}}\right)\frac{d\eta}{|\eta^{-1}\circ \xi|^{Q+s}_{\mathbb{H}^n}}
				\right] d\xi\nonumber\\
				&\le C\int_{B_{R}}G\left( \frac{w_-}{R^s}\right)d\xi
				+C||w_-||_{L^1(B_R)}\mathrm{Tail}(w_-;\xi_0,R).
			\end{align}
			Observe that $w_-\le k$ in $B_R$ by $u\ge 0$ in $B_R$. Thus we can see
			\begin{equation*}
				||w_-||_{L^1(B_R)}\le k|B_R|\quad \text{and}\quad \int_{B_{R}}G\left( \frac{w_-}{R^s}\right)d\xi\le G\left( \frac{k}{R^s}\right)|B_R|.
			\end{equation*}
			In addition, the tail is evaluated as
			\begin{equation*}
				\begin{aligned}
					\mathrm{Tail}(w_-;\xi_0,R)&\le \int_{\mathbb{H}^n\setminus B_R}g\left( \frac{u_-+k}{|\xi_0^{-1}\circ \xi|^s_{\mathbb{H}^n}}\right) \frac{d\xi}{|\xi_0^{-1}\circ \xi|^{Q+s}_{\mathbb{H}^n}}\\
					&\le C\int_{\mathbb{H}^n\setminus B_R}
					\left[
					g\left( \frac{u_-}{|\xi_0^{-1}\circ \xi|^s_{\mathbb{H}^n}}\right)+g\left( \frac{k}{|\xi_0^{-1}\circ \xi|^s_{\mathbb{H}^n}}\right)
					\right]
					\frac{d\xi}{|\xi_0^{-1}\circ \xi|^{Q+s}_{\mathbb{H}^n}}\\
					&\le C\mathrm{Tail}(u_-;\xi_0,R)+C\frac{1}{R^s}g\left( \frac{k}{R^s}\right).
				\end{aligned}
			\end{equation*}
			Thereby, the right-hand side of \eqref{3.4-1} could be controlled by
			\begin{equation*}
				C\left[
				G\left( \frac{k}{R^s}\right)+\frac{k}{R^s}g \left( \frac{k}{R^s}\right)+k\mathrm{Tail}((u_-;\xi_0,R))
				\right] \left| B_R\right| .
			\end{equation*}\par
			We in turn tackle the left-hand side of \eqref{3.4-1}. Note that for $\xi\in B_{\frac{R}{2}}(\xi_0)$ and $\eta\in \mathbb{H}^n\setminus B_R(\xi_0)$, there holds that $|\eta^{-1}\circ \xi|_{\mathbb{H}^n}\le |\eta^{-1}\circ \xi_0|_{\mathbb{H}^n}+|\xi_0^{-1}\circ \xi|_{\mathbb{H}^n}\le 2|\xi_0^{-1}\circ \eta|_{\mathbb{H}^n}$. Furthermore, by $k\ge 0$, we discover
			\begin{equation*}
				g(u_+)\le g\left( (u-k)_++k\right)\le C(p,q)\left( g\left( (u-k)_+\right)+g(k)\right).
			\end{equation*}
			Thanks to these two facts, we calculate
			\begin{equation*}
				\begin{aligned}
					&\ \quad\int_{B_{R}}w_-(\xi)\int_{\mathbb{H}^n}g\left( \frac{w_+(\eta)}{|\eta^{-1}\circ \xi|^s_{\mathbb{H}^n}}\right)\frac{d\eta}{|\eta^{-1}\circ \xi|^{Q+s}_{\mathbb{H}^n}}\,d\xi\\
					&\ge Ck\int_{B_{\frac{R}{2}}}\int_{\mathbb{H}^n\setminus B_R}g\left( \frac{w_+(\eta)}{|\xi_0^{-1}\circ \eta|^s_{\mathbb{H}^n}}\right) \frac{d\eta}{|\xi_0^{-1}\circ \eta|^{Q+s}_{\mathbb{H}^n}}\,d\xi\\
					&\ge Ck|B_R|\int_{\mathbb{H}^n\setminus B_R}\left[
					g\left( \frac{u_+}{|\xi_0^{-1}\circ \eta|^s_{\mathbb{H}^n}}\right)-g\left( \frac{k}{|\xi_0^{-1}\circ \eta|^s_{\mathbb{H}^n}}\right)
					\right] \frac{d\eta}{|\xi_0^{-1}\circ \eta|^{Q+s}_{\mathbb{H}^n}}\\
					&\ge Ck|B_R|\mathrm{Tail}(u_+;\xi_0,R)-C\frac{k}{R^s}g\left( \frac{k}{R^s}\right)|B_R|.
				\end{aligned}
			\end{equation*}
			In summary, we arrive at
			\begin{equation*}
				k\mathrm{Tail}(u_+;\xi_0,R)\le CG\left( \frac{k}{R^s}\right)+C\frac{k}{R^s}g\left( \frac{k}{R^s}\right)+Ck\mathrm{Tail}(u_-;\xi_0,R),
			\end{equation*}
			namely,
			\begin{equation*}
				\mathrm{Tail}(u_+;\xi_0,R)\le C\frac{1}{R^s}g\left( \frac{k}{R^s}\right)+C\mathrm{Tail}(u_-;\xi_0,R)
			\end{equation*}
			with $C$ depending only on $n,p,q,s$. Now following the computations in \cite[page 19]{FZ}, we deduce that
			\begin{equation*}
				R^sg^{-1}\left( R^s\mathrm{Tail}(u_+;\xi_0,R)\right)
				\le C\left( \sup_{B_R}u+R^sg^{-1}\left( R^s\mathrm{Tail}(u_-;\xi_0,R)\right)\right).
			\end{equation*}
We complete this proof now.
		\end{proof}

		In the end, we carry out the proof of the nonlocal Harnack estimate given in Theorem \ref{K}. For this purpose, we need the forthcoming result obtained by \cite{Ok}.

		\begin{lemma}
			\label{3.5}
			Assume that $F:\left[ 0,\infty\right)\rightarrow\left[ 0,\infty\right)  $ is a nondecreasing function and that the function $t\rightarrow\frac{F(t)}{t}$ is nonincreasing. Then we can find a concave function $\bar{F}:\left[ 0,\infty\right)\rightarrow\left[ 0,\infty\right)$ fulfilling
			\begin{equation*}
				\frac{1}{2}\bar{F}(t)\le F(t)\le \bar{F}(t)\quad \text{for }t\ge 0.
			\end{equation*}
		\end{lemma}

	\noindent	\textbf{Proof of Theorem \ref{K}}.
			Fix any $\eta\in B_R:=B_R(\xi_0)$ and $r\in \left( 0,2R\right] $. Then the local boundedness result, Theorem \ref{C}, gives directly
			\begin{equation*}
				\sup_{B_{r}(\eta)}u\le C(2r)^sG^{-1}\left( \delta^\gamma\mint_{B_{2r}}G\left( \frac{u}{(2r)^s}\right)d\xi \right)+(2r)^sg^{-1}\left( \delta r^s\mathrm{Tail}(u_+;\eta,r)\right),
			\end{equation*}
			where $\gamma=\frac{\theta}{1-\theta}<0$ for $\theta>1$ given by Lemma \ref{B}. Observe that, for $t\ge 0$,
			\begin{equation*}
				G^{-1}(at)\le \left( \frac{q}{p}\right)^{\frac{1}{p}}a^{\frac{1}{p}}G^{-1}(t) \quad \text{with }a\ge 1
			\end{equation*}
			and
			\begin{equation*}
				g^{-1}(at)\le \left( \frac{q}{p}\right)^{\frac{1}{p-1}}a^{\frac{1}{q-1}}g^{-1}(t) \quad \text{with }0<a\le 1.
			\end{equation*}
			Taking these facts and the tail estimate, Lemma \ref{J}, into account, we obtain
			\begin{equation*}
				\begin{aligned}
					\sup_{B_{r}(\eta)}u&\le C\delta^{\frac{\gamma}{p}}(2r)^sG^{-1}\left( \mint_{B_{2r}}G\left( \frac{u}{(2r)^s}\right)d\xi\right)  +C\delta^\frac{1}{q-1}r^sg^{-1}\left( r^s\mathrm{Tail}(u_+;\eta,r)\right)\\
					&\le C\delta^{\frac{\gamma}{p}}(2r)^sG^{-1}\left( \mint_{B_{2r}}G\left( \frac{u}{(2r)^s}\right)d\xi\right) \\
					&\quad+C\delta^\frac{1}{q-1}
					\left( \sup_{B_{r}(\eta)}u+r^sg^{-1}\left( r^s\mathrm{Tail}(u_-;\eta,r)\right)\right).
				\end{aligned}
			\end{equation*}
			Via taking such a small $\delta>0$ that $C\delta^{\frac{1}{q-1}}\le \frac{1}{2}$, the last display becomes
			\begin{equation*}
				\sup_{B_{r}(\eta)}u\le C\delta^{\frac{\gamma}{p}}(2r)^sG^{-1}\left( \mint_{B_{2r}}G\left( \frac{u}{(2r)^s}\right)d\xi\right)+Cr^sg^{-1}\left( r^s\mathrm{Tail}(u_-;\eta,r)\right).
			\end{equation*}
			Thanks to Lemma \ref{3.5}, we could find a concave function $\bar{G}(t)$ fulfilling $\bar{G}(t)\approx G(t^\frac{1}{q})$. Then an application of Jensen's inequality with $\bar{G}^{-1}(t)\approx\left( G^{-1}(t)\right)^q $ to the previous inequality indicates
			\begin{equation*}
				\sup_{B_r(\eta)}u\le C\delta^\frac{\gamma}{p}\left( \mint_{B_{2r}}u^q\,d\xi\right)^\frac{1}{q}+Cr^sg^{-1}\left( r^s\mathrm{Tail}(u_-;\eta,r)\right).
			\end{equation*}
			For the constant $\epsilon_0\in (0,1)$ from Lemma \ref{I}, employing Young's inequality we derive
			\begin{equation*}
				\begin{aligned}
					\left( \mint_{B_{2r}}u^q\,d\xi\right)^{\frac{1}{q}}&\le \left( \sup_{B_{2r}}u\right)^{\frac{q-\epsilon_0}{q}}\left( \mint_{B_{2r}}u^{\epsilon_0} \,d\xi\right)^{\frac{1}{q}}\\
					&\le \delta_1\sup_{B_{2r}}u+C(\delta_1)\left( \mint_{B_{2r}}u^{\epsilon_0}d\xi\right)^{\frac{1}{\epsilon_0}}.
				\end{aligned}
			\end{equation*}
			Thus, via further picking $\delta_1\in (0,1)$ so small that $C\delta^\frac{\gamma}{p}\delta_1=\frac{1}{2}$, the combination of two inequalities above yields
			\begin{equation*}
				\sup_{B_{r}(\eta)}u\le \frac{1}{2}\sup_{B_{2r}(\eta)}u+C\left( \mint_{B_{2r}(\eta)}u^{\epsilon_0} \,d\xi\right)^{\frac{1}{\epsilon_0}}+Cr^sg^{-1}\left( r^s\mathrm{Tail}(u_-;\eta,r)\right).
			\end{equation*}\par
			Now we would like to make use of the iteration tool, Lemma \ref{2.7Lemma}. Set $1\le \sigma<\tau\le 2$ and $r=(\tau-\sigma)R$. We conclude through a covering argument that
			\begin{equation*}
				\sup_{B_{\sigma R}}u\le \frac{1}{2}\sup_{B_{\tau R}}u
				+\frac{C}{(\tau-\sigma)^{\frac{Q}{\epsilon_0}}}\left( \mint_{B_{2R}}u^{\epsilon_0} \,d\xi\right)^\frac{1}{\epsilon_0}
				+CR^sg^{-1}\left( R^s\mathrm{Tail}(u_-;\xi_0,R)\right).
			\end{equation*}
			Therefore, apply Lemma \ref{2.7Lemma} to deduce
			\begin{equation*}
				\begin{aligned}
					\sup_{B_{R}}u&\le C\left( \frac{Q}{\epsilon_0}\right)\left[
					C\left( \mint_{B_{2R}}u^{\epsilon_0} \,d\xi\right)^\frac{1}{\epsilon_0}+CR^sg^{-1}\left( R^s\mathrm{Tail}(u_-;\xi_0,R)\right)
					\right] \\
					&\le C\left( \inf_{B_{R}}u+R^sg^{-1}\left( R^s\mathrm{Tail}(u_-;\xi_0,R)\right)\right),
				\end{aligned}
			\end{equation*}
			where in the last line Lemma \ref{I} has been used and the constant $C>0$ depends only on $n, p, q, s$.

	\section*{Data availability}
	No data was used for the research described in the article.
	
	\section*{Acknowledgments}
This work was supported by the National Natural Science Foundation of China (No. 12071098), the China Postdoctoral Science Foundation (No. 2023M730875) and the Postdoctoral Science Foundation of Heilongjiang Province (No. LBH-Z22177).

\end{document}